\documentclass{amsart}

\usepackage[hmarginratio=1:1]{geometry}
\usepackage[utf8]{inputenc}		
\usepackage{amsthm}				
\usepackage{amsmath}			
\usepackage{amsfonts}			
\usepackage{amssymb}			
\usepackage{mathtools}			
\usepackage{stmaryrd}			
\usepackage{bbm}				
\usepackage{hyperref}			

\allowdisplaybreaks[1]			

\numberwithin{equation}{section}

\newtheorem{theorem}{Theorem}[section]

\newtheorem{lemma}[theorem]{Lemma}
\newtheorem{proposition}[theorem]{Proposition}

\theoremstyle{definition}
\newtheorem{definition}[theorem]{Definition}

\theoremstyle{remark}
\newtheorem*{remark}{Remark}

\makeatletter
\DeclarePairedDelimiterX{\pmodx}[1]{(}{)}{{\operator@font mod}\mkern6mu#1}
\renewcommand{\pmod}{													
  \allowbreak															
  \if@display\mkern18mu\else\mkern8mu\fi
  \pmodx
}
\makeatother

\newcommand{\bs}{\backslash}
\newcommand{\Mod}[1]{\mkern8mu\left(#1\right)}			

\newcommand{\CC}{\mathbb C}					
\newcommand{\NN}{\mathbb N}
\newcommand{\QQ}{\mathbb Q}
\newcommand{\RR}{\mathbb R}
\newcommand{\TT}{\mathbb T}
\newcommand{\ZZ}{\mathbb Z}

\newcommand{\XXc}{\mathcal X}				

\DeclareMathOperator{\Mat}{Mat}
\DeclareMathOperator{\GL}{GL}
\DeclareMathOperator{\SL}{SL}
\DeclareMathOperator{\Sp}{Sp}

\DeclareMathOperator{\diag}{diag}			
\DeclareMathOperator{\rk}{rk}
\DeclareMathOperator{\tr}{tr}
\DeclareMathOperator{\adj}{adj}

\newcommand{\pmatrixABCD}{\begin{pmatrix}A&B\\C&D\end{pmatrix}}
\newcommand{\pmatrixtwo}[4]{\begin{pmatrix}#1&#2\\#3&#4\end{pmatrix}}

\newcommand{\smmatrix}[4]{\left(\begin{smallmatrix}#1&#2\\#3&#4\end{smallmatrix}\right)}

\newcommand{\smmatrixABCD}{\left(\begin{smallmatrix}A&B\\C&D\end{smallmatrix}\right)}

\newcommand{\norm}[1]{\left\lVert #1\right\rVert}		
\newcommand{\abs}[1]{\left\lvert #1\right\rvert}
\newcommand{\divides}{\mid}
\newcommand{\ndivides}{\nmid}

\usepackage[
	backend=biber,
	style=alphabetic,
	maxnames=99,
	maxalphanames=5,
	doi=false,
	isbn=false
]{biblatex}

\renewbibmacro{in:}{}								
\DeclareFieldFormat{pages}{#1}
\DeclareFieldFormat{extraalpha}{#1}
\AtEveryBibitem{\clearfield{eprintclass}}

\DeclareLabelalphaTemplate{					
  \labelelement{
    \field[final]{shorthand}
    \field{label}
    \field[strwidth=3,strside=left,ifnames=1]{labelname}
    \field[strwidth=1,strside=left]{labelname}
  }
}
\begin{document}
\subjclass{11L05,11L07}
\keywords{Kloosterman sums, symplectic group, Siegel modular forms}

\author{Gilles Felber}
\address{Alfréd Rényi Institute of Mathematics, Reáltanoda street 13-15, H-1053, Budapest}
\email{felber@renyi.hu}

\title{Symplectic Kloosterman Sums for $\Sp(2n)$ with Powerful Moduli}
\begin{abstract}
We prove a non-trivial bound for $\Sp(2n)$ Kloosterman sums of moduli not equal to a prime multiple of the identity. These sums are attached to Siegel modular forms on the group $\Sp(2n)$ and appear in the corresponding Petersson formula. We give an application to equidistribution of coprime symmetric pairs.
\end{abstract}
\maketitle

\section{Introduction}
Kloosterman sums are a type of exponential sums that play a significant role in number theory. They allow for multiple generalizations over various groups such as $\GL(n)$ and $\Sp(2n)$. The generalizations appear in particular in relative trace formulas of Petersson/Kuznetsov type and in Fourier coefficients of Poincaré series, but also in relation to equidistribution problems. Recently, non-trivial bounds have been proved for Kloosterman sums over groups of higher ranks. Blomer-Man and Linn \cite{BM22,Lin24} considered the Kloosterman sums appearing in the Kuznetsov formula for $\GL(n)$. Erdélyi, Tóth and Zábrady \cite{ET24,ETZ24} considered another type of $\GL(n)$ Kloosterman sums appearing in equidistribution problems \cite{ELS22}. For the Petersson formula of the symplectic group, only the case $\Sp(4)$ was considered until now with non-trivial bounds proven by Kitaoka and Tóth \cite{Kit84,Tot13}. In the Kuznetsov formula for $\Sp(4)$, the sums were bounded by Man \cite{Man21}.

In this paper, we consider a generalization of Kloosterman sums to $\Sp(2n)$ appearing in the theory of Siegel modular forms and in the corresponding Petersson formula. Let $n$ be an integer, $C\in\Mat_n(\ZZ)$ a matrix with $\det(C)\neq0$ and $Q$ and $T$ be two symmetric half-integral matrices. The symplectic Kloosterman sum is
\begin{align}\label{eq definition of SKS}
K_n(Q,T;C)=\sum_{\smmatrixABCD\in X(C)}e^{2\pi i\tr(AC^{-1}Q+C^{-1}DT)}.
\end{align}
The sum is over symplectic matrices in the double quotient
$$X(C):=\Gamma_\infty\bs\left\{\pmatrixtwo ABCD\in\Sp_{2n}(\ZZ)\right\}/\Gamma_\infty$$
with $\Gamma_\infty=\{\smmatrix{I_n}X0{I_n}\in\Sp_{2n}(\ZZ)\}$. To simplify, we write $e(M):=e^{2\pi i\tr(M)}$ for a square matrix $M$. For $n=1$, this is consistent with the usual notation in number theory. Since $\Sp_2(\RR)=\SL_2(\RR)$, we obtain the classical Kloosterman sum in that case. We have the celebrated Weil bound \cite{Wei48}
$$\abs{K_1(q,t;c)}=\abs{\sum_{x\Mod c,\ (x,c)=1}e(c^{-1}qx+c^{-1}t\bar x)}\leq\tau(c)(c,q,t)^{1/2}c^{1/2}.$$
In this introduction, we consider $n\geq2$. Since $e(M)=1$ for a matrix $M\in\Mat_n(\ZZ)$, summing over $X(C)$ is well defined. Unless necessary, we drop the size of the matrices $n$ from the notation.

For any $C\in\Mat_n(\ZZ)$ with $\det(C)\neq0$, we have $U,V\in\GL_n(\ZZ)$ such that $UCV=\diag(c_1,\dots,c_n)$ with $c_1\divides\cdots\divides c_n$. The integers $c_1,\dots,c_n$ are called the elementary divisors of $C$ and they are unique. The diagonal matrix is called the Smith normal form of $C$. We show in Section \ref{sec elementary properties} that the dependency of $K(Q,T;C)$ in $C$ is only in its elementary divisors. In particular, we have the trivial bound
\begin{align}\label{eq trivial bound}
K_n(Q,T;C)\leq\prod_{i=1}^n c_i^{n-i+1}.
\end{align}
In the scalar case, when $C=m I_n$, the trivial bound is $m^{n(n+1)/2}$. Moreover, we can factorize the sum with respect to the prime numbers dividing $c_n$.

Let $p$ be a prime and consider $C$ of the form $\diag(p^{\sigma_1},\dots,p^{\sigma_n})$ with $0\leq\sigma_1\leq\dots\leq\sigma_n$. Our first result is a non-trivial bound for Kloosterman sums when at least $\sigma_n\geq2$. In the following theorem, the notation $(a,M,N)$ for integral matrices $M$ and $N$ means the greatest common divisor of $a$ and all the coordinates in $M$ and $N$.

\begin{theorem}\label{thm bound for higher prime powers}
Let $p$ be a prime number and $Q,T$ be two symmetric half-integral matrices.
\begin{enumerate}
\item Let $C=p^\sigma I_n$ be a scalar matrix with $\sigma\geq2$. Let $\sigma=2\mu+\nu$ with $\nu=0$ if $\sigma$ is even and 1 otherwise. Then
$$K_n(Q,T;C)\ll_n p^{\sigma n^2/2}(p^\mu,2Q,2T)^n(p^\nu,2Q,2T)^{n/2}.$$
\item Let $C=\diag(p^{\sigma_1},\dots,p^{\sigma_n})$ with $0\leq\sigma_1\leq\dots\leq\sigma_n$ and $\sigma_n\geq2$. Let $\sigma_i=2\mu_i+\nu_i$ with $\nu_i=0$ if $\sigma_i$ is even and 1 otherwise.
$$K_n(Q,T;C)\ll_n\prod_{i=1}^np^{(n-i+1/2)\sigma_i}(p^{\mu_i},2Q_i')(p^{\nu_i},2Q_i')^{1/2}.$$
Here $Q_i'$ is defined as follow:
	\begin{enumerate}
		\item If $\sigma_i=1$, then $Q_i'$ is the right block of $Q$ of size $n$ by $s$, where $s$ is the smallest integer with $\sigma_s>1$. In that case, $\mu_i=0$.
	\item If $\sigma_i\geq2$, then $Q_i'$ is the bottom-right block of $Q$ of size $s$ by $s$, where $s$ is the smallest integer with $\sigma_s=\sigma_i$.
	\end{enumerate}
\end{enumerate}
In both cases, the implicit constant only depends on the dimension $n$.
\end{theorem}

In Lemma \ref{lem symmetry in QT}, we show that the roles of $Q$ and $T$ can be interchanged in the second bound. A bound similar to Theorem \ref{thm bound for higher prime powers} was proven by Márton, Tóth and Zábrádi in the case $C=pI_n$.

\begin{theorem}[\cite{MT26,TZ25}, to appear]\label{thm bound Toth and Zabradi}
Let $p$ be an odd prime number and $C=pI_n$. We have
$$K_n(Q,T;C)\ll p^{n(n+1)/2-r/2}$$
with $r=\max\{\rk_pQ,\rk_pT\}$ where the ranks are taken modulo $p$.
\end{theorem}

In their articles, they also show that their result is essentially optimal. A non-trivial bound for $K_n(Q,T;C)$ was first proven by Kitaoka \cite{Kit84} for $n=2$. Later, Tóth proved square-root cancellation for these sums \cite{Tot13}, which is the best possible bound. Of course, for $n=1$, the Weil bound already gives square-root cancellation. With the result from Márton, Tóth and Zábrádi, this is the first non-trivial bound for symplectic Kloosterman where $n\geq3$. The symplectic Kloosterman sums appear in many applications. In particular, Kitaoka introduced them, in the paper cited above, to bound Fourier coefficients of Siegel modular forms. We hope to return to the generalization of these questions to $\Sp_{2n}(\RR)$ in the near future.

An important proof strategy for us is to decompose the modulus $C$ into blocks of constant prime powers. It leads to induction on the number of blocks and reduction of problems on coprime symmetric pairs with various congruences to problems on symmetric matrices and a unique congruence. This is used in particular in Sections \ref{sec block decomposition} and \ref{sec matrix Gauss sums}. By a coprime symmetric pair, we mean the two bottom blocks $(C,D)$ of a symplectic matrix. See Proposition \ref{pro equivalences coprime symmetric pair} for equivalent definitions.

The proof of Theorem \ref{thm bound for higher prime powers} is essentially in three parts. For $C$ diagonal consisting of powers of $p$, we can suppose that $A=\bar D^t$ is the inverse transpose matrix of $D\pmod{p^{\sigma_n}}$ for $\smmatrixABCD\in X(C)$. The proof starts by a $p$-adic stationary phase argument in Section \ref{sec Taylor argument} for $C$ consisting of prime powers larger than $p$. The challenge here is to combine the multiplicative structure of $\bar D^t$ with its additive structure, given by the fact $CD^t$ is symmetric for a symplectic matrix. Then in Section \ref{sec block decomposition}, we split $C$ into two blocks: $C=\diag(pI_s,C_1)$ with prime powers in $C_1$ larger than $p$. We split in the same way all the other matrices appearing in the sum. After computing the block inverse, we restructure the sum and can insert the results of the last section. The final result is given in Proposition \ref{pro section block decomposition}. The symplectic Kloosterman sum is now given by a Kloosterman sum with $C=pI_s$, a quadratic matrix equation and two quadratic Gauss sums over matrices modulo respectively $p\Mat_{s,n-s}(\ZZ)$ and $C'\Mat_{n-s}(\ZZ)$ with the elements in $C'$ equal to 1 or $p$, whether the corresponding prime power is even or odd. In the second sum, there is an additional symmetry condition on the summed matrices. Finally, in Section \ref{sec matrix Gauss sums}, we prove non-trivial bounds over the two Gauss sums and the number of solutions to the quadratic equation. This relies in particular on a block decomposition of $C_1$ into different prime powers and a list of simpler matrix equations, for which we show non-trivial bounds. We prove our theorem without appealing to Theorem \ref{thm bound Toth and Zabradi} thanks to the first Gauss sum modulo $p$, that correspond to the top-right block of $D$. The bound for this sum gives us a large enough win over the trivial bound for the "$p$-part" of $C$ corresponding to its first block. The case $p=2$ is treated at the end of the section.

In this article, we develop a robust framework that allows for a square-root cancellation bound with additional efforts. One would need to give better bounds to the matrix equations in Lemma \ref{lem simpler matrix equations} and in Case 1 of the proof of Proposition \ref{pro symmetric sum over matrices}, compute the Gauss sum of Proposition \ref{pro Gauss sum over matrices} exactly (this was done by Walling, see the remark after the statement) and compute the resulting sum in $W$ of Proposition \ref{pro section block decomposition}, which is a slightly modified symplectic Kloosterman sum modulo $p$. The improved bounds will depend on the rank of various blocks of the parameters $Q$ and $T$. Thus the non-generic bound will be quite technical to state and use. In any case, this would be limited in applications without a corresponding bound for a sum over $C=pI_n$ as in Theorem \ref{thm bound Toth and Zabradi}.

In Section \ref{sec applications}, we give an application of Theorem \ref{thm bound for higher prime powers} in the spirit of an article of El-Baz, Lee and Strömbergsson \cite{ELS22}. Sums over a general $C$ can be factorized with respect to the divisors of its elementary divisors. This is detailed in Section \ref{sec elementary properties}. We combine Theorems \ref{thm bound for higher prime powers} and \ref{thm bound Toth and Zabradi} to get a general bound. Then we apply it to the following equidistribution problem. Let $\TT_n=\XXc_n(\RR/\ZZ)$ be the set of $n$ by $n$ symmetric matrices modulo 1. Let $C\in\Mat_n(\ZZ)$ be such that $\det(C)\neq0$. Consider
$$S_C:=\left\{(C^{-t}A^t,C^{-1}D)\in\TT_n\times\TT_n\,\middle|\,\pmatrixtwo A\ast CD\in X(C)\right\}.$$

\begin{theorem}\label{thm application intro}
Let $C_0\in\Mat_n(\ZZ)$ be such that $\det(C_0)\neq0$ and $m\in\NN$. The set $S_{mC_0}$ equidistributes effectively in $\TT_n\times\TT_n=\XXc_n(\RR/\ZZ)^2$ as $m\to\infty$.
\end{theorem}
A more precise statement with an explicit rate of convergence is given in Section \ref{sec applications}. The case $n=1$ was presented in \cite{EMSS15}.

\subsection{Notations}
We denote the set of $n$ by $n$ symmetric matrices by $\XXc_n$. If needed, we precise the ring in parenthesis. Half-integral symmetric matrices are elements of $\XXc(\RR)$ with half-integral coefficients and integral diagonal. They correspond to quadratic forms.

Let $M$ be a square matrix. We write $e(M):=e^{2\pi i\tr(M)}$. Note that for a 1 by 1 matrix, this is consistent with the notation frequently used in number theory. If $M$ is invertible, we write $M^{-t}$ for the transpose of the inverse of $M$. For two square matrices $M,N$, we write $M[N]:=N^tMN$. Let $p$ be a prime number and $M\in\Mat_n(\ZZ)$ be an integral matrix (or a half-integral matrix if $p\neq2$). We write $\rk_p(M)$ for the rank of the reduction modulo $p$ of the matrix $M$. We write $0_n$ for the $n$ by $n$ matrix with only zeros. 

We will consider (half-)integral matrices modulo various sets. Since matrix multiplication is non-commutative, we will always precise the full set for the reduction. For example, for a matrix $C$, we write $[C]:=C\Mat_n(\ZZ)+\Mat_n(\ZZ)C$. We will consider matrices modulo $[C]$ in Section \ref{sec matrix Gauss sums}.

We write $(a_,\dots,a_r)$ to denote the greatest common divisor between $a_1,\dots,a_r$. If some $a_i$ is replaced by an integral matrix, we mean by this notation the greatest common divisor of all the coordinates in the matrix and the rest of the $a_j$. We write $a\divides b$ to denote that $a$ divides $b$ and $(p^\infty,a)$ to denote the largest power of $p$ that divides $a$. We use the Vinogradov symbols $\ll$ and $\gg$, with index to precise the dependency of the implicit constant if needed.

\subsection{Acknowledgment}
The author thanks Valentin Blomer and Árpád Tóth for their help and guidance on this project. The research towards this paper was supported by the MTA–RI Lendület “Momentum” Analytic Number Theory and Representation Theory Research Group.

\section{Elementary properties}\label{sec elementary properties}

\subsection{Symplectic matrices}
Let $n$ be a positive integer. The \emph{symplectic group} is
$$\Sp_{2n}(\RR):=\{M\in\Mat_{2n}(\RR)\mid M^tJM=J\}$$
with $J=\smmatrix0{I_n}{-I_n}0$. Unless stated otherwise, we always split elements of the symplectic group in $n$ by $n$ blocks. We write $\Sp_{2n}(\ZZ)$ for the set of elements of $\Sp_{2n}(\RR)$ that have integral entries. We write
$$\Gamma_\infty:=\{\smmatrix{I_n}X0{I_n}\mid X\in\Mat_n(\ZZ)\text{ symmetric}\}\subseteq\Sp_{2n}(\ZZ).$$

\begin{lemma}\label{lem equivalences symplectic matrix}
Let $M=\smmatrixABCD\in\Mat_{2n}(\RR)$ be a matrix. The following are equivalent:
\begin{enumerate}
\item$M$ is symplectic.
\item$A^tC$ and $B^tD$ are symmetric and $A^tD-C^tB=I_n$.
\item$AB^t$ and $CD^t$ are symmetric and  $DA^t-CB^t=I_n$.
\end{enumerate}
Moreover, suppose that $\det(C)\neq0$. Then $M$ is symplectic if and only if $A^tC$ and $CD^t$ are symmetric and $DA^t-CB^t=I_n$.
\end{lemma}

\begin{proof}
The first equivalences are direct consequences of the definition. For the last statement, we only need to check that $AB^t$ is symmetric. Using the hypothesis above, we see that $B^t=C^{-1}(DA^t-I_n)$ and that $AC^{-1}$ and $C^{-1}D$ are symmetric. Then
$$AB^t=AC^{-1}(DA^t-I_n)=AD^tC^{-t}A^t-C^{-t}A^t=(AD^t-I_n)C^{-t}A^t=BA^t.$$
\end{proof}

\begin{remark}
Suppose we are given matrices $A,C,D$ with $C$ invertible and $A^tC$ and $CD^t$ symmetric. There is a unique way to complete the blocks to a symplectic matrix $\smmatrixABCD$ by setting $B=(AD^t-I_n)C^{-t}$. Moreover, if all the matrices are integral and $AD^t=I_n\pmod{\Mat_n(\ZZ)C^t}$, we get an integral symplectic matrix $\smmatrixABCD$.
\end{remark}

\begin{definition}
Let $C\in\Mat_{m,n}(\ZZ)$ and $r=\rk(C)$. There exists matrices $U\in\GL_m(\ZZ),V\in\GL_n(\ZZ)$ such that
$$UCV=\pmatrixtwo{C'}{}{}{0_{n-r}}$$
with $C'=\diag(c_1,\dots,c_r)$ and $c_1\divides c_2\divides\cdots\divides c_r$. The matrix $UCV$ is called the \emph{Smith normal form} of $C$ and the positive integers $c_1,\dots,c_r$ are called the \emph{elementary divisors} of $C$. They are unique. We have the formula
$$d_i(C)=c_1\cdots c_i$$
where $d_i(C)$ is the greatest common divisor of all minors of size $i$ in $C$.
\end{definition}

\begin{definition}
A \emph{symmetric pair} $(C,D)$ consists of two integral matrices such that $CD^t$ is symmetric. A \emph{coprime symmetric pair} $(C,D)$ consists of two integral matrices that are the bottom line of an integral symplectic matrix $\smmatrix **CD\in\Sp_{2n}(\ZZ)$.
\end{definition}

\begin{proposition}\label{pro equivalences coprime symmetric pair}
Let $C,D$ be two square integral matrices of size $n$. The following are equivalent:
\begin{enumerate}
\item$(C,D)$ is a coprime symmetric pair.
\item$(D,C)$ is a coprime symmetric pair.
\item$CD^t$ is symmetric and for all $G\in\GL_n(\QQ)$, $G\begin{pmatrix}C&D\end{pmatrix}$ is integral if and only if $G\in\GL_n(\ZZ)$.
\item $CD^t$ is symmetric and the greatest common divisor of all the minors of size $n$ in $\begin{pmatrix}C&D\end{pmatrix}$ is 1.
\end{enumerate}
\end{proposition}

\begin{proof}
$(1)\Leftrightarrow(2)$: $\pmatrixABCD$ is symplectic if and only if $\pmatrixtwo{-B}{-A}DC$ is symplectic.\\
$(3)\Leftrightarrow(4)$: let $\begin{pmatrix}F&0\end{pmatrix}$ be the Smith normal form of $\begin{pmatrix}C&D\end{pmatrix}$. That is, there exists $U\in\GL_n(\ZZ)$, $V\in\GL_{2n}(\ZZ)$ such that $U\begin{pmatrix}C&D\end{pmatrix}V=\begin{pmatrix}F&0\end{pmatrix}$. Let $G\in\GL_n(\QQ)$. Then
$$G\begin{pmatrix}C&D\end{pmatrix}\ \text{is integral}\Leftrightarrow GU^{-1}F\ \text{is integral}.$$
Note that (4) is equivalent to $F=I_n$. If (4) holds, then $G$ must be integral, so (3) holds. Conversely, if (4) does not hold, then $f_{nn}\neq1$. Then the matrix $G=UF^{-1}$ is not integral and contradicts (3).\\
$(1)\Rightarrow(3)$: if $\pmatrixABCD$ is an integral symplectic matrix, then
$$G=G\cdot I_n=GDA^t-GCB^t.$$
Clearly, if $GC$ and $GD$ are integral, then $G$ is integral.\\
$(3)\Rightarrow(1)$: See \cite{Sie35}, Lemma 42. Alternatively $(4)\Rightarrow(1)$ is proven in \cite{New85}.
\end{proof}

\begin{lemma}\label{lem characterizations coprime symmetric pair}\ 
\begin{enumerate}
\item Let $p$ be a prime. If $C$ is a diagonal matrix of the form $\diag(p^{\sigma_1},\dots,p^{\sigma_n})$ with $\sigma_1,\dots,\sigma_n\geq1$, then $(C,D)$ is a coprime symmetric pair if and only if $C^tD$ is symmetric and $p\ndivides\det(D)$.
\item If $\smmatrixABCD$ is a symplectic matrix, then $(A,B)$, $(A^t,C^t)$, $(B^t,D^t)$ and $(C,D)$ are coprime symmetric pairs.
\end{enumerate}
\end{lemma}

\begin{proof}\ 
\begin{enumerate}
\item Let $(C,D)$ be a symmetric pair. By Proposition \ref{pro equivalences coprime symmetric pair} (4), $(C,D)$ is a coprime symmetric pair if and only if one of its $n$ by $n$ minors is coprime to $p$. Since any minor with a column of $C$ is divisible by $p$, it is necessary (and sufficient) that $(\det(D),p)=1$.
\item This is clear since if $\smmatrixABCD$ is a symplectic matrix, then so are
$$\pmatrixtwo{-C}{-D}AB,\quad\pmatrixtwo{-B^t}{-D^t}{A^t}{C^t},\quad\pmatrixtwo{A^t}{C^t}{B^t}{D^t}.$$
\end{enumerate}
\end{proof}

We can characterize elements of the double quotient
$$X(C):=\Gamma_\infty\bs\{\smmatrixABCD\in\Sp_{2n}(\ZZ)\}/\Gamma_\infty.$$
\begin{proposition}\label{pro double quotient characterization}
Let $C$ be an invertible matrix. The set $X(C)$ is in bijection with
$$\tilde X(C):=\{D\pmod{C\Mat_n(\ZZ)}\mid (C,D)\text{ coprime symmetric pair}\},$$
by sending a matrix to its bottom-right block $D$.
\end{proposition}

\begin{remark}
This proves Equation \eqref{eq trivial bound}.
\end{remark}

\begin{proof}
Clearly $(C,D)$ must always form a coprime symmetric pair. We check which pairs are in the same class in $X(C)$. Suppose first that $\det(C)\neq0$. We show that if we have two matrices with equal bottom blocks,
$$\pmatrixtwo{A_1}{B_1}CD,\pmatrixtwo{A_2}{B_2}CD,$$
then they are equivalent in $\Gamma_\infty\bs\Sp_{2n}(\RR)$. We have
$$A_1D^t-B_1C^t=I_n,\ A_2D^t-B_2C^t=I_n\quad\Rightarrow\quad(A_1-A_2)D^t=(B_1-B_2)C^t.$$
So $B_1-B_2=(A_1-A_2)D^tC^{-t}=(A_1-A_2)C^{-1}D$. Then
$$(A_1-A_2)C^{-1}\begin{pmatrix}C&D\end{pmatrix}=\begin{pmatrix}A_1-A_2&B_1-B_2\end{pmatrix}\in\Mat_{n,2n}(\ZZ).$$
By Proposition \ref{pro equivalences coprime symmetric pair}, $X=(A_1-A_2)C^{-1}\in\Mat_n(\ZZ)$ and we also have $XD=B_1-B_2$. Note also that $X$ is symmetric. Then
$$\pmatrixtwo{I_n}X{}{I_n}\pmatrixtwo{A_1}{B_1}CD=\pmatrixtwo{A_1+XC}{B_1+XD}CD=\pmatrixtwo{A_2}{B_2}CD.$$

Now, let $\smmatrixABCD\in\Sp_{2n}(\ZZ)$ and $X_1,X_2\in\XXc_n(\ZZ)$. We have
$$\pmatrixtwo{I_n}{X_1}{}{I_n}\pmatrixABCD\pmatrixtwo{I_n}{X_2}{}{I_n}=\pmatrixtwo**C{CX_2+D}.$$
So two matrices are in the same class in $X(C)$ if and only if their bottom-right block is equal $\pmod{C\Mat_n(\ZZ)}$.
\end{proof}

\subsection{Factorization of Kloosterman sums}
In this section, we reduce the study of the Kloosterman sum to the case where $C$ is a diagonal matrix consisting only of prime powers. The two following lemmas are generalizations to $n\geq2$ of Lemmas 1, 2 and 3 in \cite{Kit84}. The proof is similar, except for the bijection $f$ in Lemma \ref{lem factorization prime powers}, where we give more details.

\begin{lemma}\label{lem factorization Smith normal form} 
Let $C\in\Mat_n(\ZZ)$ be an invertible matrix and $Q,T$ be two symmetric half-integral matrices. Let $U,V\in\GL_n(\ZZ)$. Then
$$K_n(Q,T;C)=K_n(Q[U],T[V];U^tCV).$$
\end{lemma}

\begin{proof}
Let $X(C)$ be the double quotient as above. Then
$$\pmatrixtwo{U^{-1}}{}{}{U^t}X(C)\pmatrixtwo V{}{}{V^{-t}}=X(U^tCV).$$
More precisely, the above equation defines a bijection between the two sets. This is because the matrices $\smmatrix{U^{-1}}{}{}{U^t}$ and $\smmatrix V{}{}{V^{-t}}$ normalize $\Gamma_\infty$ since $U,V\in\GL_n(\ZZ)$ and because
$$\pmatrixtwo{U^{-1}}{}{}{U^t}\pmatrixABCD\pmatrixtwo V{}{}{V^{-t}}=\pmatrixtwo{U^{-1}AV}{U^{-1}BV^{-t}}{U^tCV}{U^tDV^{-t}}.$$
Since this identity can be reversed, we have a bijection between $X(C)$ and $X(U^tCV)$. The lemma is then established by invariance of the trace under conjugation:
\begin{align*}
K(Q[U],T[V];U^tCV)&=\sum_{\smmatrixABCD\in X(U^tCV)}e(A(U^tCV)^{-1}Q[U]+(U^tCV)^{-1}DT[V])\\
	&=\sum_{\smmatrixABCD\in X(C)}e((U^{-1}AV)(U^tCV)^{-1}Q[U]+(U^tCV)^{-1}(U^tDV^{-t})T[V])\\
	&=\sum_{\smmatrixABCD\in X(C)}e(AC^{-1}Q+C^{-1}DT).
\end{align*}
\end{proof}

\begin{lemma}\label{lem factorization prime powers}
Let $C=FG\in\Mat_n(\ZZ)$ be invertible diagonal matrices in Smith normal form with $(f_{nn},g_{nn})=1$. Let $r,s\in\ZZ$ be such that $rf_{nn}+sg_{nn}=1$. Let $Q,T$ be two symmetric half-integral matrices. Then
$$K_n(Q,T;C)=K_n(Q[\bar G],T;F)\cdot K_n(Q[\bar F],T;G),$$
where $\bar F=rf_{nn}F^{-1}$ and $\bar G=sg_{nn}G^{-1}$.
\end{lemma}

\begin{proof}
Note that $\bar FF+\bar GG=I_n$. First, we show that we have a bijection $f:X(C)\to X(F)\times X(G)$ given by
\begin{align}\label{eq bijection X(C) and X(F) times X(G)}
\pmatrixABCD\mapsto\left[\pmatrixtwo{GA}{GB-\bar FA^tD}F{\bar GD},\pmatrixtwo{FA}{FB-\bar GA^tD}G{\bar FD}\right].
\end{align}
Suppose that $\smmatrixABCD\in\Sp_{2n}(\ZZ)$. Then $A^tGF=A^tC$ and $(B^tG-D^tA\bar F)\bar GD=sg_{nn}B^tD-rsc_{nn}D^tAC^{-1}D$ are symmetric matrices. We used that $C,F$ and $G$ are diagonal. Moreover
$$A^tG\bar GD-F(GB-\bar FA^tD)=sg_{nn}A^tD-CB+rf_{nn}A^tD=A^tD-C^tB=I_n.$$
So the first component on the right-hand side of Equation \eqref{eq bijection X(C) and X(F) times X(G)} is in $\Sp_{2n}(\ZZ)$. The second is as well, since we can exchange the roles of $F$ and $G$. Conversely suppose that the right-hand side of Equation \eqref{eq bijection X(C) and X(F) times X(G)} is in $\Sp_{2n}(\ZZ)\times\Sp_{2n}(\ZZ)$. To construct the inverse, consider a pair
$$\left(\pmatrixtwo{A_F}{B_F}F{D_F},\pmatrixtwo{A_G}{B_G}G{C_G}\right)\in\Sp_{2n}(\ZZ)\times\Sp_{2n}(\ZZ)$$
It inverse image $\smmatrixABCD$ is given by
\begin{align*}
A&=\bar GA_F+\bar FA_G,\qquad
C=FG,\qquad
D=GD_F+FD_G,\\
B&=2\bar F\bar GA^tD+\bar GB_F+\bar FB_G.
\end{align*}
It is clear that this is an inverse for $f$. We need to check that $f$ and its inverse are well defined, i.e. they factor through the double quotient. By Proposition \ref{pro double quotient characterization}, we only have to check the bottom lines of the maps. We have
$$\pmatrixtwo\ast\ast CD\pmatrixtwo{I_n}X{}{I_n}=\pmatrixtwo\ast\ast C{D+CX}\mapsto\left[\pmatrixtwo\ast\ast F{\bar GD+\bar GCX},\pmatrixtwo\ast\ast G{\bar FD+\bar FCX}\right].$$
Since $\bar GCX=sg_{nn}FX$ and $\bar FCX=rf_{nn}GX$ are multiples of $FX$ respectively $GX$, theses are matrices equivalent to the images of $\smmatrix\ast\ast CD$. Conversely, we have
$$\left[\pmatrixtwo\ast\ast F{D_F+FX},\pmatrixtwo\ast\ast G{D_G+GY}\right]\mapsto\pmatrixtwo\ast\ast{FG}{GD_F+FD_G+FG(X+Y)}.$$
The image is clearly in $\smmatrix\ast\ast{FG}{GD_F+FD_G}\Gamma_\infty$.

Therefore $f$ has an inverse and is injective. Since $X(C)$ and $X(F)\times X(G)$ are finite, it suffices to show that they have the same cardinality to show that $f$ is bijective. Let $\tilde X(C)$ be as in Proposition \ref{pro double quotient characterization}. We have a function $g:\tilde X(C)\mapsto\tilde X(F)\times\tilde X(G)$ given by $D\mapsto(\bar GD,\bar FD)$. This is a restriction of $f$ to the bottom-right block. Its inverse is
$$(D_F,D_G)\mapsto GD_F+FD_G.$$
If $(C,D)$ is a coprime symmetric pair, so is $(\bar GC,\bar GD)$. Conversely, let $D=GD_F+FD_G$. Clearly $CD^t$ is symmetric. We need to show that $(C,D)$ is a coprime symmetric pair. We show that the rank modulo $p$ of $\begin{pmatrix}C&D\end{pmatrix}$ is $n$ for all primes $p$. Meaning that the greatest common divisor of all minors of size $n$ in $\begin{pmatrix}C&D\end{pmatrix}$ is coprime to $p$. By Proposition \ref{pro equivalences coprime symmetric pair}, this is equivalent. If we multiply $\begin{pmatrix}C&D\end{pmatrix}$ on the left or on the right by a matrix $M$ with $p\ndivides\det(M)$, then the rank modulo $p$ does not change. Suppose that $p\ndivides g_{nn}$. Then
$$\rk_p\begin{pmatrix}C&D\end{pmatrix}=\rk_p\left(\bar G\begin{pmatrix}FG&GD_F+FD_G\end{pmatrix}\pmatrixtwo{I_n}{-\bar GD_G}{}{I_n}\right)=\rk_p\begin{pmatrix}F&D_F\end{pmatrix}=n.$$
If $p\divides g_{nn}$, we do the same with $F$ instead of $G$. This show that $g$ is a bijection and that $X(C)$ and $X(F)\times X(G)$ have the same cardinality. Therefore $f$ is also a bijection.

Finally, we compute
\begin{align*}
K&(Q,T;C)=\sum_{\smmatrixABCD\in X(C)}e(AC^{-1}Q+C^{-1}DT)\\
	&=\sum_{\smmatrix{A_F}{B_F}F{D_F}\in X(F)}\sum_{\smmatrix{A_G}{B_G}G{D_G}\in X(G)}e((\bar GA_F+\bar FA_G)(FG)^{-1}Q+(FG)^{-1}(GD_F+FD_G)T)\\
	&=\sum_{\smmatrix{A_F}{B_F}F{D_F}\in X(F)}e(\bar GA_FF^{-1}G^{-1}Q+F^{-1}D_FT)\\
	&\quad\cdot\sum_{\smmatrix{A_G}{B_G}G{D_G}\in X(G)}e(\bar FA_GF^{-1}G^{-1}Q+G^{-1}D_GT)\\
	&=\sum_{\smmatrix{A_F}{B_F}F{D_F}\in X(F)}e(\bar GA_FF^{-1}G^{-1}(\bar FF+\bar GG)Q+F^{-1}D_FT)\\
	&\quad\cdot\sum_{\smmatrix{A_G}{B_G}G{D_G}\in X(G)}e(\bar FA_GF^{-1}G^{-1}(\bar FF+\bar GG)Q+G^{-1}D_GT)\\
	&=\sum_{\smmatrix{A_F}{B_F}F{D_F}\in X(F)}e(A_FF^{-1}Q[\bar G]+F^{-1}D_FT)\sum_{\smmatrix{A_G}{B_G}G{D_G}\in X(G)}e(A_GG^{-1}Q[\bar F]+G^{-1}D_GT)\\
	&\quad\cdot e(\bar GA_FG^{-1}\bar FQ+\bar FA_GF^{-1}\bar GQ)
\end{align*}
In the final line, since $A_F=GA$ and $A_G=FA$, we have
$$e(\bar GA_FG^{-1}\bar FQ+\bar FA_GF^{-1}\bar GQ)=e(sg_{nn}AG^{-1}\bar FQ+rf_{nn}AF^{-1}\bar GQ)=e(2A\bar F\bar GQ)=1.$$
We conclude that
$$K(Q,T;C)=K(Q[\bar G],T;F)\cdot K(Q[\bar F],T;G).$$
\end{proof}

\begin{lemma}\label{lem symmetry in QT}
Let $C\in\Mat_n(\ZZ)$ be an invertible matrix and $Q,T$ be two symmetric half-integral matrices. Then
$$K_n(Q,T;C)=K_n(T,Q;C^t).$$
\end{lemma}

\begin{proof}
Note that $\smmatrixABCD\in\Sp_{2n}(\ZZ)$ if and only if $\smmatrix{D^t}{B^t}{C^t}{A^t}\in\Sp_{2n}(\ZZ)$. This defines a bijection $X(C)\to X(C^t)$ because for a symmetric matrix $X\in\XXc(\ZZ)$, we have
\begin{align*}
\pmatrixtwo{I_n}X{}{I_n}\pmatrixABCD&=\pmatrixtwo{A+XC}{B+XD}CD&&\mapsto&\pmatrixtwo{D^t}{B^t}{C^t}{A^t}\pmatrixtwo{I_n}X{}{I_n},\\
\pmatrixABCD\pmatrixtwo{I_n}X{}{I_n}&=\pmatrixtwo A{AX+B}C{CX+D}&&\mapsto&\pmatrixtwo{I_n}X{}{I_n}\pmatrixtwo{D^t}{B^t}{C^t}{A^t}.
\end{align*}
Then we get
\begin{align*}
K(Q,T;C^t)&=\sum_{\smmatrix AB{C^t}D\in X(C^t)}e(AC^{-t}Q+C^{-t}DT)\\
	&=\sum_{\smmatrixABCD\in X(C)}e(D^tC^{-t}Q+C^{-t}A^tT)\\
	&=\sum_{\smmatrixABCD\in X(C)}e(C^{-1}DQ+AC^{-1}T)\\
	&=K(Q,T;C).
\end{align*}
\end{proof}

Applying Lemmas \ref{lem factorization Smith normal form} and \ref{lem factorization prime powers}, we can reduce to a matrix $C$ of the shape
$$C=\begin{pmatrix}p^{\sigma_1}\\&p^{\sigma_2}\\&&\ddots\\&&&p^{\sigma_n}\end{pmatrix}$$
with $0\leq\sigma_1\leq\dots\leq\sigma_n$. We make this more precise in the proof of Proposition \ref{pro final bound}. With such a $C$, the set $X(C)$ is in bijection with
\begin{align}\label{eq definition tilde X(C)}
\tilde X(C):=\left\{D=\begin{pmatrix}d_{11}&d_{12}&\cdots&d_{1n}\\p^{\sigma_2-\sigma_1}d_{12}&d_{22}&\cdots&d_{2n}\\\vdots&\vdots&\ddots&\vdots\\p^{\sigma_n-\sigma_1}d_{1n}&p^{\sigma_n-\sigma_2}d_{2n}&\cdots&d_{nn}\end{pmatrix}\middle|\ \begin{array}{c}\forall i\leq j\\d_{ij}\pmod{p^{\sigma_i}},\\(\det(D),p)=1\end{array}\right\}
\end{align}
If $\sigma_i=0$, we fix $d_{ii}=1$ and $d_{ij}=0$ for $i<j$. This is direct from Proposition \ref{pro double quotient characterization} and Lemma \ref{lem characterizations coprime symmetric pair}. Therefore
$$\abs{X(C)}\leq\prod_{i=1}^np^{(n-i+1)\sigma_i}.$$
This prove the trivial bound \eqref{eq trivial bound} in the case of a matrix $C$ of the shape $\diag(p^{\sigma_1},\dots,p^{\sigma_n})$.

\begin{definition}\label{def D bar}
Given a matrix $D\in\tilde X(C)$, we define $\bar D=d\adj(D)$ where $d\det(D)=1\pmod{p^{\sigma_n}}$ and $\adj(D)$ is the adjugate matrix of $D$. Then $\bar DD=D\bar D=I_n\pmod{p^{\sigma_n}\Mat_n(\ZZ)}$. In particular, the matrix $B=C^{-1}(\bar DD-I_n)$ is integral and we have
$$\pmatrixtwo{\bar D^t}BCD\in X(C).$$
Note also that $(C,\bar D)$ is also a coprime symmetric pair. Finally, we conclude that
\begin{align}\label{eq SKS with D bar}
K(Q,T;C)=\sum_{\smmatrixABCD\in X(C)}e(AC^{-1}Q+C^{-1}DT)=\sum_{D\in\tilde X(C)}e(C^{-1}\bar DQ+C^{-1}DT).
\end{align}
\end{definition}

To close this section, we consider a degenerated case.

\begin{proposition}\label{pro sigma_1=0}
Let $C=\diag(I_s,p^{\sigma_{s+1}},\dots,p^{\sigma_n})$ with $\sigma_{s+1}\geq1$. Let $Q$ and $T$ be half-integral symmetric matrices. Let $Q_3,T_3,C_3$ be the $n-s$ by $n-s$ bottom-right block of respectively $Q,T,C$. Then
$$K_n(Q,T;C)=K_{n-s}(Q_3,T_3;C_3).$$
If $s=n$, we interpret $K_0$ as 1.
\end{proposition}

\begin{proof}
Recall that if $D\in\tilde X(C)$ with $C$ as in the proposition, then $d_{ij}=\delta_{ij}$ for $i\leq s$ or $j\leq s$. We have a bijection $\tilde X(C_3)\mapsto\tilde X(C)$ given by
$$D_3\mapsto\pmatrixtwo{I_s}{}{}{D_3}.$$
Moreover, if $D_3A_3^t=I_{n-s}\pmod{C_3\Mat_n(\ZZ)}$, then
$$\pmatrixtwo{I_s}{}{}{D_3}\pmatrixtwo{I_s}{}{}{A_3}^t=I_n\pmod{C\Mat_n(\ZZ)}.$$
Then
$$\tr\left(\pmatrixtwo{I_s}{}{}{A_3}C^{-1}\pmatrixtwo{Q_1}{Q_2}{Q_2^t}{Q_3}+C^{-1}\pmatrixtwo{I_s}{}{}{D_3}\pmatrixtwo{T_1}{T_2}{T_2^t}{T_3}\right)=\tr\left(A_3C_3^{-1}Q_3+C_3^{-1}D_3T_3\right)$$
and $K_n(Q,T;C)=K_{n-s}(Q_3,T_3;C_3)$.
\end{proof}

\section{Taylor expansion argument}\label{sec Taylor argument}
In this section, we consider the case $C=\diag(p^{\sigma_1},\dots,p^{\sigma_n})$ with $2\leq\sigma_1\leq\dots\leq\sigma_n$. Our goal is to prove Proposition \ref{pro section Taylor argument}. The proof's strategy is to do a finite Taylor expansion of the coefficients with respect to smaller prime powers (also known as $p$-adic stationary phase). Recall Equation \eqref{eq definition tilde X(C)} and Definition \ref{def D bar}. Let $\mu_i=\lfloor\frac{\sigma_i}2\rfloor$ for $i=1,\dots,n$ and $\tilde C=\diag(p^{\mu_1},\dots,p^{\mu_n})$. Given a fixed $D_1\in\tilde X(C)$, we consider all the $D\in\tilde X(C)$ such that $D=D_1\pmod{\tilde C\Mat_n(\ZZ))}$. We write $D=D_1+\tilde CD_2$. Clearly, we have
$$\tilde CD_2C=(D-D_1)C=C(D^t-D_1^t)=CD_2^t\tilde C.$$
So $(C\tilde C^{-1},D_2)$ is a symmetric pair. Let $\bar D_1^t$ as in Definition \ref{def D bar}. Our first goal is to construct, given $C,D_1,\bar D_1$, a symplectic matrix $\smmatrixABCD$ from any $D=D_1\pmod{\tilde C\Mat_n(\ZZ)}$ such that $(C,D)$ is a symmetric pair and with an explicit formula for $A=\bar D_1^t$. First, we prove a small but very useful lemma.

\begin{lemma}\label{lem integrality of G-1HG}
Let $F=\diag(p^{a_1},\dots,p^{a_n})$ and $G=\diag(p^{b_1}\dots,p^{b_n})$ be two diagonal matrices with increasing prime powers on the diagonal and $H$ be an integral matrix such that $(F,H)$ is a symmetric pair. If for all $i\geq j$
$$b_i-b_j\leq a_i-a_j,$$
then $G^{-1}HG$ is an integral matrix.
\end{lemma}

\begin{proof}
Since $FH^t=HF$, the matrix $F^{-1}HF=H^t$ is integral. That is, for $i\geq j$
$$p^{a_i-a_j}\divides h_{ij}.$$
Since $b_i-b_j\leq a_i-a_j$, we have that $G^{-1}HG$ is integral as well.
\end{proof}

\begin{lemma}\label{lem congruence powers of tilde CD2A1t}
Let $D_1,D_2\in\Mat_n(\ZZ)$ be such that $(C,D_1)$ is a coprime symmetric pair and $(C\tilde C^{-1},D_2)$ is a symmetric pair. We have the following congruences:
\begin{align*}
p(\tilde CD_2\bar D_1)^2&=0\pmod{C\Mat_n(\ZZ)},\\
(\tilde CD_2\bar D_1)^3&=0\pmod{C\Mat_n(\ZZ)}.
\end{align*}
\end{lemma}

\begin{proof}
Recall that $(C,\bar D_1)$ is also a coprime symmetric pair. For all $i\geq j$, we have
$$\mu_i-\mu_j\leq\sigma_i-\sigma_j\Leftrightarrow\left\lceil\frac{\sigma_j}2\right\rceil\leq\left\lceil\frac{\sigma_i}2\right\rceil.$$
By Lemma \ref{lem integrality of G-1HG}, the matrix $\tilde C^{-1}\bar D_1\tilde C$ is integral. Moreover $(C\tilde C^{-1},D_2)$ is a symmetric pair, so
$$D_2\bar D_1\tilde C=D_2\tilde C(\tilde C^{-1}\bar D_1\tilde C)=C\tilde C^{-1}D_2^t(C^{-1}\tilde C^2)(\tilde C^{-1}\bar D_1\tilde C).$$
We get
$$(\tilde CD_2\bar D_1)^2=CD_2^t(C^{-1}\tilde C^2)(\tilde C^{-1}\bar D_1\tilde C)D_2\bar D_1.$$
Note that everything in the above equation is integral except the product $C^{-1}\tilde C^2$. The $k$-th coefficient of the diagonal matrix $pC^{-1}\tilde C^2$ is $2\mu_k-\sigma_k+1\geq0$, so the matrix is integral. In conclusion, for some matrix $M\in\Mat_n(\ZZ)$, we have
$$p(\tilde CD_2\bar D_1)^2=CM$$
and the congruence holds. For the second equation in the proposition, since every power of $p$ in $\tilde C$ is at least one, we can write
$$(\tilde CD_2\bar D_1)^3=p(\tilde CD_2\bar D_1)^2(p^{-1}\tilde CD_2\bar D_1).$$
By the first equation, this is in $C\Mat_n(\ZZ)$.
\end{proof}

\begin{lemma}\label{lem construction A from D_1}
Let $D_1,D_2\in\Mat_n(\ZZ)$ be such that $(C,D_1)$ is a coprime symmetric pair and $(C\tilde C^{-1},D_2)$ is a symmetric pair. Write $D=D_1+\tilde CD_2$ and
$$A^t=\bar D_1(I_n-\tilde CD_2\bar D_1+(\tilde CD_2\bar D_1)^2).$$
Then $B^t=C^{-1}(DA^t-I_n)$ is an integral matrix and the matrix
$$\pmatrixtwo ABCD\in\Sp_{2n}(\ZZ).$$
\end{lemma}

\begin{proof}
By the discussion at the start of this section, we know that $(C,D)$ is a symmetric pair. By Lemma \ref{lem equivalences symplectic matrix}, we need to prove two things: (1) $A^tC=C^tA$. (2) $B^t$ is integral. First, since $(C,\bar D_1)$ and $(C\tilde C^{-1},D_2)$ are symmetric pairs, we have
$$\tilde CD_2\bar D_1C=\tilde CD_2C\bar D_1^t=CD_2^t\tilde C\bar D_1^t.$$
Therefore
$$A^tC=\bar D_1(I_n-\tilde CD_2\bar D_1+(\tilde CD_2\bar D_1)^2)C=C\bar D_1^t(I_n-D_2^t\tilde C\bar D_1^t+(D_2^t\tilde C\bar D_1^t)^2)=C^tA.$$
Next, recall that $D_1\bar D_1=I_n+CB_1^t$. We compute
\begin{align*}
DA^t&=(D_1+\tilde CD_2)\bar D_1(I_n-\tilde CD_2\bar D_1+(\tilde CD_2\bar D_1)^2)\\
	&=(I_n+\tilde CD_2\bar D_1)(I_n-\tilde CD_2\bar D_1+(\tilde CD_2\bar D_1)^2)+CB_1^t(I_n-\tilde CD_2\bar D_1+(\tilde CD_2\bar D_1)^2)\\
	&=(I_n+(\tilde CD_2\bar D_1)^3)+CB_1^t(I_n-\tilde CD_2\bar D_1+(\tilde CD_2\bar D_1)^2)
\end{align*}
By Lemma \ref{lem congruence powers of tilde CD2A1t}, $(\tilde CD_2\bar D_1)^3=0\pmod{C\Mat_n(\ZZ)}$. In conclusion, $B^t$ is integral since
$$DA^t=I_n\pmod{C\Mat_n(\ZZ)}.$$
\end{proof}

Let $D_1\in\tilde X(C)$. We consider the set
$$\{D\in\tilde X(C)\mid D=D_1\pmod{\tilde C\Mat_n(\ZZ)}\}.$$
By the discussion at the start of the section, the matrices in this set are exactly the matrices of the form $D_1+\tilde CD_2$ for a matrix $D_2\pmod{C\tilde C^{-1}\Mat_n(\ZZ)}$ such that $(C\tilde C^{-1},D_2)$ is a symmetric pair. By Lemma \ref{lem construction A from D_1}, we have a bijective map
\begin{align*}
\tilde X_1(C)\times\tilde X_2(C)&\to\tilde X(C),\\
(D_1,D_2)&\mapsto D_1+\tilde CD_2
\end{align*}
with $\tilde X_1(C)$ a set of representatives of the matrices in $\tilde X(C)$ modulo $\tilde C\Mat_n(\ZZ)$ and
$$\tilde X_2(C)=\{D_2\pmod{C\tilde C^{-1}\Mat_n(\ZZ)}\mid(C\tilde C^{-1},D_2)\text{ symmetric pair}\}.$$
Moreover, given a pair $(D_1,D_2)$, a corresponding symplectic matrix in $X(C)$ is $\smmatrix A{*}C{D_1+\tilde CD_2}$ with
$$A^t=\bar D_1(I_n-\tilde CD_2\bar D_1+(\tilde CD_2\bar D_1)^2).$$
We can now split $K(Q,T;C)$ into two sums over $\tilde X_1(C)$ and $\tilde X_2(C)$. We have
{\allowdisplaybreaks[4]	
\begin{align}\label{eq sum over X_1 and X_2}
K(Q,T;C)&=\sum_{D\in\tilde X(C)}e(C^{-1}A^tQ+C^{-1}DT)\nonumber\\
	&=\sum_{D_1\in\tilde X_1(C)}\sum_{D_2\in\tilde X_2(C)}e(C^{-1}\bar D_1(I_n-\tilde CD_2\bar D_1+(\tilde CD_2\bar D_1)^2)Q+C^{-1}(D_1+\tilde CD_2)T)\nonumber\nonumber\\
	&=\sum_{D_1\in\tilde X_1(C)}e(C^{-1}\bar D_1Q+C^{-1}D_1T)\nonumber\\
	&\quad\cdot\sum_{D_2\in\tilde X_2(C)}e(C^{-1}\bar D_1(-\tilde CD_2\bar D_1+(\tilde CD_2\bar D_1)^2)Q+C^{-1}\tilde CD_2T)\nonumber\\
	&=\sum_{D_1\in\tilde X_1(C)}e(C^{-1}\bar D_1Q+C^{-1}D_1T)\nonumber\\
	&\quad\cdot\sum_{D_2\in\tilde X_2(C)}e(C^{-1}\tilde CD_2(T-\bar D_1Q\bar D_1^t)+\bar D_1^tC^{-1}(\tilde CD_2\bar D_1)^2Q).
\end{align}}
On the last line, we used again that is a symmetric pair. Note that
$$p\bar D_1^tC^{-1}(\tilde CD_2\bar D_1)^2Q\in\Mat_n(\ZZ)$$
by Lemma \ref{lem congruence powers of tilde CD2A1t}. We can be more precise. By Lemma \ref{lem integrality of G-1HG}, the matrices $\tilde C^{-1}\bar D_1\tilde C$ and $C^{-1}\tilde C\bar D_1C\tilde C^{-1}$ are integral. We compute
\begin{align}\label{eq cancellation for D_22}
C^{-1}(\tilde CD_2\bar D_1)^2&=D_2^tC^{-1}\tilde C\bar D_1\tilde CD_2\bar D_1\nonumber\\
	&=D_2^t(C^{-1}\tilde C^2)(\tilde C^{-1}\bar D_1\tilde C)D_2\bar D_1\nonumber\\
	&=D_2^t(C^{-1}\tilde C\bar D_1C\tilde C^{-1})(C^{-1}\tilde C^2)D_2\bar D_1.
\end{align}
We write $D_2=D_{2,1}+C\tilde C^{-2}D_{2,2}$. By Equation \eqref{eq cancellation for D_22}, we have
$$e(\bar D_1^tC^{-1}(\tilde CD_2\bar D_1)^2Q)=e(\bar D_1^tC^{-1}(\tilde CD_{2,1}\bar D_1)^2Q),$$
because any factor with $D_{2,2}$ is integral.

We have a bijection $D_2\mapsto D_{2,1}+C\tilde C^{-2}D_{2,2}$ between the summand sets $\tilde X_2(C)\to\tilde X_{2,1}(C)\times\tilde X_{2,2}(C)$ with
\begin{align*}
\tilde X_{2,1}(C)&=\{D_{2,1}\pmod{C\tilde C^{-2}\Mat_n(\ZZ)}\mid(C\tilde C^{-1},D_{2,1})\text{ symmetric pair}\},\\
\tilde X_{2,2}(C)&=\{D_{2,2}\pmod{\tilde C\Mat_n(\ZZ)}\mid(\tilde C,D_{2,2})\text{ symmetric pair}\}.
\end{align*}
Therefore, we showed that in Equation \eqref{eq sum over X_1 and X_2}, the sum over $\tilde X_2(C)$ can be rewritten as
$$\sum_{D_{2,1}\in\tilde X_{2,1}(C)}e(C^{-1}\tilde CD_{2,1}(T-\bar D_1Q\bar D_1^t)+\bar D_1^tC^{-1}(\tilde CD_{2,1}\bar D_1)^2Q)\sum_{D_{2,2}\in\tilde X_{2,2}(C)}e(\tilde C^{-1}D_{2,2}(T-\bar D_1Q\bar D_1^t)).$$
Suppose here that $p\neq2$. Applying Lemma \ref{lem symmetric character sum over matrices} that we will prove in Section \ref{sec matrix Gauss sums}, the sum over $\tilde X_{2,2}(C)$ vanishes unless
$$T=\bar D_1Q\bar D_1^t\pmod{[\tilde C]}$$
where $[\tilde C]=\tilde C\Mat_n(\ZZ)+\Mat_n(\ZZ)\tilde C$. If the congruence holds, then all the terms are 1 and the sum is
$$\abs{\tilde X_{2,2}(C)}=\prod_{i=1}^np^{(n-i+1)\mu_i}$$
In conclusion, we showed the following in this section.

\begin{proposition}\label{pro section Taylor argument}
Let $p$ be an odd prime. Let $C=\diag(p^{\sigma_1},\dots,p^{\sigma_n})$ be a matrix with $2\leq\sigma_1\leq\dots\leq\sigma_n$. Following the notation introduced in this section, we have
\begin{align*}
K(Q,T;C)&=\prod_{i=1}^np^{(n-i+1)\mu_i}\sum_{\substack{D_1\in\tilde X_1(C)\\T=\bar D_1Q\bar D_1^t\Mod{[\tilde C]}}}e(C^{-1}\bar D_1Q+C^{-1}D_1T)\\
	&\quad\cdot\sum_{D_{2,1}\in\tilde X_{2,1}(C)}e(C^{-1}\tilde CD_{2,1}(T-\bar D_1Q\bar D_1^t)+\bar D_1^tC^{-1}(\tilde CD_{2,1}\bar D_1)^2Q).
\end{align*}
\end{proposition}

In this section, we consider the case $C=\diag(p^{\sigma_1},\dots,p^{\sigma_n})$ with $2\leq\sigma_1\leq\dots\leq\sigma_n$. Our goal is to prove Proposition \ref{pro section Taylor argument}. The proof's strategy is to do a finite Taylor expansion of the coefficients with respect to smaller prime powers (also known as $p$-adic stationary phase). Recall Equation \eqref{eq definition tilde X(C)} and Definition \ref{def D bar}. Let $\mu_i=\lfloor\frac{\sigma_i}2\rfloor$ for $i=1,\dots,n$ and $\tilde C=\diag(p^{\mu_1},\dots,p^{\mu_n})$. Given a fixed $D_1\in\tilde X(C)$, we consider all the $D\in\tilde X(C)$ such that $D=D_1\pmod{\tilde C\Mat_n(\ZZ))}$. We write $D=D_1+\tilde CD_2$. Clearly, we have

\section{Block decomposition}\label{sec block decomposition}
In this section, we consider $C=\diag(p^{\sigma_1},\dots,p^{\sigma_n})$ with $\sigma_1\geq1$ and $\sigma_n\geq2$. Our goal is generalizing the results of Section \ref{sec Taylor argument} to the case $\sigma_1=1$. The idea is to write
$$C=\pmatrixtwo{pI_s}{}{}{C_1}$$
with all the prime powers in $C_1$ being at least $p^2$. For this section, we fix $s$ as the larger index such that $\sigma_s=1$. We decompose all the matrices appearing in our sum with respect to the blocks of $C$, that is
$$Q=\pmatrixtwo{Q_1}{Q_2}{Q_2^t}{Q_3},\quad T=\pmatrixtwo{T_1}{T_2}{T_2^t}{T_3}$$
with $Q_1$ and $T_1$ blocks of size $s$ by $s$. For a matrix $D\in\tilde X(C)$, we write
$$D=\pmatrixtwo WXYZ$$
with $W$ a block of size $s$ by $s$. Since $(C,D)$ is a coprime symmetric pair, we have
\begin{align}\label{eq symmetries of WXYZ}
\pmatrixtwo{pW}{XC_1}{pY}{ZC_1}=DC=CD^t=\pmatrixtwo{pW^t}{pY^t}{C_1X^t}{C_1Z^t}.
\end{align}
In particular, $W$ is symmetric and $Y=p^{-1}C_1X^t$ is equal to $0\pmod{p\Mat_{n-s,s}(\ZZ)}$. Thus $\det(D)=\det(W)\det(Z)\pmod p$ and $p\ndivides\det(W),\det(Z)$. We also see that $(pI_s,W)$ and $(C_1,Z)$ are coprime symmetric pairs.

We want to compute the inverse of $D$ modulo a high-enough prime power.

\begin{lemma}\label{lem inverse of WXYZ}
Let $(C,D)$ be a coprime symmetric pair with $D=\smmatrix WXYZ\in\Mat_n(\ZZ)$, where $W$ a block of size $s$ by $s$. Let $U=Z-Y\bar WX$ be the Schur complement of $W$. Then $(C_1,U)$ is a coprime symmetric pair and
$$A^t=\pmatrixtwo{\bar W+\bar WX\bar UY\bar W}{-\bar WX\bar U}{-\bar UY\bar W}{\bar U}$$
is such that $DA^t=I_n\pmod{C\Mat_n(\ZZ)}$. Moreover $(A^t,C)$ is a symmetric pair and, with $B^t=(DA^t-I_n)C^{-1}$, we have
$$\pmatrixABCD\in\Sp_{2n}(\ZZ).$$
Here, $\bar W$ and $\bar U$ are as in Definition \ref{def D bar}.
\end{lemma}

\begin{remark}
The shape of $A^t$ comes from the inversion formula for block matrices.
\end{remark}

\begin{proof}
Since $Y=p^{-1}C_1X^t=0\pmod{p\Mat_{n-s,s}(\ZZ)}$ by Equation \eqref{eq symmetries of WXYZ}, we have
$$\det(U)=\det(Z)\neq0\pmod{p\Mat_{n-s}(\ZZ)}.$$
Note that if $W$ is symmetric, so is $\bar W$. By Equation \eqref{eq symmetries of WXYZ}, we get
$$UC_1=(Z-Y\bar WX)C_1=C_1(Z^t-C_1X^t\bar W^tY^t)=C_1U^t.$$
So $(C_1,U)$ is a coprime symmetric pair. Let $\bar U$ be as in Definition \ref{def D bar}. We have two symplectic matrices
$$\pmatrixtwo{\bar W^t}{B_W}{pI_s}{W}\in\Sp_{2s}(\ZZ),\quad\pmatrixtwo{\bar U^t}{B_U}{C_1}U\in\Sp_{2(n-s)}(\ZZ).$$
In particular $W\bar W=I_s+pB_W^t$ and $U\bar U=I_{n-s}+C_1B_U^t$. We compute
\begin{align*}
DA^t&=\pmatrixtwo WXY{U+Y\bar WX}\pmatrixtwo{\bar W+\bar WX\bar UY\bar W}{-\bar WX\bar U}{-\bar UY\bar W}{\bar U}\\
	&=\smmatrix{I_s+pB_W^t+(I_s+pB_W^t)X\bar UY\bar W-X\bar UY\bar W}{-(I_s+pB_W^t)X\bar U+X\bar U}{Y\bar W+Y\bar WX\bar UY\bar W-(I_{n-s}+C_1B_U^t)Y\bar W-Y\bar WX\bar UY\bar W}{-Y\bar WX\bar U+I_{n-s}+C_1B_U^t+Y\bar WX\bar U}\\
	&=\pmatrixtwo{I_s+pB_W^t(I_s+X\bar UY\bar W)}{-pB_W^tX\bar U}{-C_1B_U^tY\bar W}{I_{n-s}+C_1B_U^t}\\
	&=I_n\pmod{C\Mat_n(\ZZ)}.
\end{align*}

By Lemma \ref{lem equivalences symplectic matrix}, we only need to show that $(A^t,C)$ is a symmetric pair to conclude the proof of the theorem. Applying Equation \eqref{eq symmetries of WXYZ}, we get
$$A^tC=\pmatrixtwo{p(\bar W+\bar WX\bar UY\bar W)}{-\bar WX\bar UC_1}{-p\bar UY\bar W}{\bar UC_1}=\pmatrixtwo{p(\bar W^t+\bar W^tY^t\bar U^tX^t\bar W^t)}{-p\bar W^tY^t\bar U^t}{-C_1\bar U^tX^t\bar W^t}{C_1\bar U^t}=CA.$$
\end{proof}

Let $D=\smmatrix WXYZ\in\tilde X(C)$ be a matrix modulo $C\Mat_n(\ZZ)$ such that $(C,D)$ is a coprime symmetric pair. Then $W$, $X$ and $Z$ are matrices modulo respectively $p\Mat_n(\ZZ)$, $p\Mat_{s,n-s}(\ZZ)$ and $C_1\Mat_{n-s}(\ZZ)$ and $Y$ is determined by $Y=p^{-1}C_1X^t$. It is easy to see that $U=Z-Y\bar WX$ is also a matrix modulo $C_1\Mat_{n-r}(\ZZ)$. We saw at the beginning of this section that $(pI_r,W)$ and $(C_1,U)$ are coprime symmetric pair. Thus we get matrices $W\in\tilde X(p I_s)$ and $U\in\tilde X(C_1)$. We get a bijection
\begin{align*}
\tilde X(C)&\to\tilde X(pI_s)\times\Mat_{s,n-s}(\ZZ/p\ZZ)\times\tilde X(C_1),\\
\pmatrixtwo WXYZ&\mapsto(W,X,Z-Y\bar WX),
\end{align*}
with the inverse map given by Lemma \ref{lem inverse of WXYZ}:
$$(W,X,U)\mapsto\pmatrixtwo WX{p^{-1}C_1X^t}{U+Y\bar WX}.$$

We can now compute the Kloosterman sum with respect to these sets. Using $Y=p^{-1}C_1X^t$, we have
\begin{align*}
&K(Q,T;C)=\sum_{D\in\tilde X(C)}e(C^{-1}\bar DQ+C^{-1}DT)\\
	&=\sum_{W\in\tilde X(pI_s)}\sum_{X\Mod{p\Mat_{s,n-s}(\ZZ)}}\sum_{U\in\tilde X(C_1)}e\left(\pmatrixtwo{p^{-1}I_s}{}{}{C_1^{-1}}\pmatrixtwo{\bar W+\bar WX\bar UY\bar W}{-\bar WX\bar U}{-\bar UY\bar W}{\bar U}\pmatrixtwo{Q_1}{Q_2}{Q_2^t}{Q_3}\right.\\
	&\quad\left.+\pmatrixtwo{p^{-1}I_s}{}{}{C_1^{-1}}\pmatrixtwo WXY{U+Y\bar WX}\pmatrixtwo{T_1}{T_2}{T_2^t}{T_3}\right)\\
	&=\sum_{W,X,U}e\left(p^{-1}[\bar WQ_1+\bar WX\bar UY\bar WQ_1-\bar WX\bar UQ_2^t+WT_1+XT_2^t]\right)\\
	&\quad\cdot e\left(C_1^{-1}[-\bar UY\bar WQ_2+\bar UQ_3+YT_2+UT_3+Y\bar WXT_3]\right).
\end{align*}
Note that the size of the matrices is not the same on the last two lines. Since all the prime power in $C_1$ are at least $p^2$, we have
$$e(p^{-1}\bar WX\bar UY\bar WQ_1)=e(\bar WX\bar U(p^{-2}C_1)X^t\bar WQ_1)=1.$$
Note also that, since $\bar W$ is symmetric, we have
\begin{align*}
e(-p^{-1}\bar WX\bar UQ_2^t)&=e(-\bar WY^tC_1^{-1}\bar UQ_2^t)=e(-C_1^{-1}\bar UQ_2^t\bar WY^t),\\
e(C_1^{-1}YT_2)&=e(p^{-1}X^tT_2)=e(p^{-1}XT_2^t).
\end{align*}
In conclusion, we have
\begin{align}\label{eq sum WXU after simplification}
K(Q,T;C)&=\sum_{W,X,U}e\left(p^{-1}[\bar WQ_1+WT_1+2XT_2^t]\right)\nonumber\\
	&\quad\cdot e\left(C_1^{-1}[-\bar U(Y\bar WQ_2+Q_2^t\bar WY^t)+\bar UQ_3+UT_3+Y\bar WXT_3]\right).
\end{align}
The sum in $U$ is a Kloosterman sum, similar to the one in the last section. The matrix $C_1$ has all its prime powers larger or equal to $p^2$ and
$$S_U:=\sum_{U\in\tilde X(C_1)}e(C_1^{-1}\bar U(Q_3-Y\bar WQ_2-Q_2^t\bar WY^t)+C_1^{-1}UT_3)=K_{n-s}(\tilde Q,T_3;C_1)$$
with $\tilde Q=Q_3-(Y\bar WQ_2+Q_2^t\bar WY^t)$ (note that this is symmetric). Applying Proposition \ref{pro section Taylor argument} and adapting the notation there, we get
\begin{align*}
S_U&=\prod_{i=s+1}^np^{(n-i+1)\mu_i}\sum_{\substack{U_1\in\tilde X_1(C_1)\\T_3=\bar U_1\tilde Q\bar U_1^t\Mod{[\tilde C_1]}}}e(C_1^{-1}\bar U_1\tilde Q+C_1^{-1}U_1T_3)\\
	&\quad\cdot\sum_{U_{2,1}\in\tilde X_{2,1}(C_1)}e(C_1^{-1}\tilde C_1U_{2,1}(T_3-\bar U_1\tilde Q\bar U_1^t)+\bar U_1^tC_1^{-1}(\tilde CU_{2,1}\bar U_1)^2\tilde Q).
\end{align*}
Since $Y=p^{-1}C_1X^t$ and $(C_1\bar U_1)$ is a coprime symmetric pair, we have
$$\bar U_1\tilde Q\bar U_1^t=\bar U_1Q_3\bar U_1-p^{-1}C_1\bar U_1^tX^t\bar WQ_2\bar U_1^t-p^{-1}\bar U_1Q_2^t\bar WX\bar U_1C_1.$$
Note that
$$p^{-1}C_1\bar U_1^tX^t\bar WQ_2\bar U_1^t+p^{-1}\bar U_1Q_2^t\bar WX\bar U_1^tC_1\in[\tilde C_1]$$
since the prime powers in $C_1$ are all at least $p^2$. Using that $(C_1\tilde C_1^{-1},U_{2,1})$ is a symmetric pair and Lemma \ref{lem congruence powers of tilde CD2A1t}, we can also replace $\tilde Q$ by $Q_3$ in the second line of the equation for $S_U$. We get
\begin{align*}
S_U&=\prod_{i=s+1}^np^{(n-i+1)\mu_i}\sum_{\substack{U_1\in\tilde X_1(C_1)\\T_3=\bar U_1Q_3\bar U_1^t\Mod{[\tilde C_1]}}}e(C_1^{-1}\bar U_1Q_3+C_1^{-1}U_1T_3)e(-C_1^{-1}\bar U_1(Y\bar WQ_2+Q_2^t\bar WY^t))\\
	&\quad\cdot\sum_{U_{2,1}\in\tilde X_{2,1}(C_1)}e(C_1^{-1}\tilde C_1U_{2,1}(T_3-\bar U_1Q_3\bar U_1^t)+\bar U_1^tC_1^{-1}(\tilde CU_{2,1}\bar U_1)^2Q_3).
\end{align*}
Finally, we insert $S_U$ in Equation \eqref{eq sum WXU after simplification}. We get
\begin{align*}
K(Q,T;C)&=\sum_{W\in\tilde X(pI_s)}e(p^{-1}\bar WQ_1+p^{-1}WT_1)\sum_{X\Mod{p\Mat_{s,n-s}(\ZZ)}}e(2p^{-1}XT_2^t)e(C_1^{-1}Y\bar WXT_3)\cdot S_U\\
	&=\prod_{i=s+1}^np^{(n-i+1)\mu_i}\sum_{W\in\tilde X(pI_s)}e(p^{-1}\bar WQ_1+p^{-1}WT_1)\\
	&\quad\cdot\sum_{\substack{U_1\in\tilde X_1(C_1)\\T_3=\bar U_1Q_3\bar U_1^t\Mod{[\tilde C_1]}}}e(C_1^{-1}\bar U_1Q_3+C_1^{-1}U_1T_3)\\
	&\quad\cdot\sum_{X\Mod{p\Mat_{s,n-s}(\ZZ)}}e(2p^{-1}XT_2^t+XT_3Y\bar W)e(-2C_1^{-1}\bar U_1Q_2^t\bar WY^t)\\
	&\quad\cdot\sum_{U_{2,1}\in\tilde X_{2,1}(C_1)}e(C_1^{-1}\tilde C_1U_{2,1}(T_3-\bar U_1Q_3\bar U_1^t)+\bar U_1^tC_1^{-1}(\tilde CU_{2,1}\bar U_1)^2Q_3).
\end{align*}
We used that $e(-C_1^{-1}\bar UY\bar WQ_2)=e(-C_1^{-1}\bar UQ_2^t\bar WY^t)$. Replacing $Y$ by $p^{-1}C_1X^t$ in every occurrence, we proved the following.

\begin{proposition}\label{pro section block decomposition}
Let $p$ be an odd prime. Let $C=\diag(p^{\sigma_1},\dots,p^{\sigma_n})$ with $\sigma_1=\dots=\sigma_s=1$ and $2\leq\sigma_{s+1}\leq\dots\leq\sigma_n$. Following the notation introduced in this section and the one before, we have:
\begin{align*}
K(Q,T;C)&=\prod_{i=s+1}^np^{(n-i+1)\mu_i}\sum_{W\in X(pI_s)}e(p^{-1}\bar WQ_1+p^{-1}WT_1)\\
	&\quad\cdot\sum_{\substack{U_1\in\tilde X_1(C_1)\\ T_3=\bar U_1Q_3\bar U_1^t\Mod{[\tilde C_1]}}}e(C_1^{-1}\bar U_1Q_3+C_1^{-1}U_1T_3)\\
	&\quad\cdot\sum_{X\Mod{p\Mat_{s,n-s}(\ZZ)}}e(2p^{-1}X(T_2^t-\bar U_1Q_2^t\bar W)+p^{-1}XT_3X^t\bar W)\\
	&\quad\cdot\sum_{U_{2,1}\in\tilde X_{2,1}(C_1)}e(C_1^{-1}\tilde C_1U_{2,1}(T_3-\bar U_1Q_3\bar U_1^t)+\bar U_1^tC_1^{-1}(\tilde C_1U_{2,1}\bar U_1)^2Q_3).
\end{align*}
\end{proposition}

\begin{remark}
If $s=0$, we make the convention that the sums over $W$ and $U_1$ are equal to 1. Then the above result generalizes Proposition \ref{pro section Taylor argument}.
\end{remark}

We close this section by considering the degenerate case where all coefficients are divisible by $p$.

\begin{proposition}\label{pro common divisor bound}
Let $Q$ and $T$ be half-integral symmetric matrices such that $(Q,T,p)\neq1$. Let $C=\diag(pI_s,C_3)$ with prime powers in $C_3$ larger than $p$ (potentially $s=0$). Then
\begin{align*}
K_n(Q,T;C)&=\abs{X(pI_s)}\cdot p^{(n-s)(n+s+1)/2}\cdot K_{n-s}(p^{-1}Q_3,p^{-1}T_3;p^{-1}C_3)\\
	&=\abs{X(pI_s)}\cdot p^{(n-s)(n+s+1)/2}\cdot K_n(p^{-1}Q,p^{-1}T;p^{-1}C)
\end{align*}
(with $\abs{X(pI_0)}=1$.)

More generally, let $\sigma_{i,k}=\min\{\sigma_i,k\}$ and $C_{(k)}=\diag(p^{\sigma_{1,k}},\dots,p^{\sigma_{n,k}})$. Consider the largest $k$ such that $C_{(k)}^{-1}Q$ and $C_{(k)}^{-1}T$ are integral. Let $s$ be the largest index such that $\sigma_s\leq k$. Then
$$\abs{K_n(Q,T;C)}\ll\prod_{i=1}^np^{(n-i+1)\sigma_{i,k}}\cdot\abs{K_{n-s}(p^{-k}Q_3,p^{-k}T_3;p^{-k}C_3)}$$
where $Q_3,T_3,C_3$ are the bottom-right blocks of size $n-s$ by $n-s$ of resp. $Q,T,C$. If $s=n$, we interpret $K_0$ as 1.
\end{proposition}

\begin{proof}
By Equation \eqref{eq sum WXU after simplification}, following the notation there and using that $Y=p^{-1}C_1X^t$, we have
\begin{align*}
K_n(Q,T;C)&=\sum_{W,X,U}e(p^{-1}[\bar WQ_1+WT_1+2XT_2^t])\\
	&\quad\cdot e(C_1^{-1}[-2\bar UY\bar WQ_2+\bar UQ_3+UT_3+Y\bar WXT_3])\\
	&=\sum_{W,X,U}e((pC_1^{-1})\bar U(p^{-1}Q_3)+(pC_1^{-1})U(p^{-1}T_3))\\
	&=\abs{X(pI_s)}\cdot p^{s(n-s)}\cdot\sum_{U\in\tilde X(C_1)}e((pC_1^{-1})\bar U(p^{-1}Q_3)+(pC_1^{-1})U(p^{-1}T_3)).
\end{align*}
Note that $\tilde X(p^{-1}C_1)$ is a set of representatives for the matrices in $\tilde X(C_1)$ modulo $(p^{-1}C_1)\Mat_{n-s}(\ZZ)$. We write $U=U_1+p^{-1}C_1U_2$ with $U_1\in\tilde X(p^{-1}C_1)$ and
$$U_2\in\tilde X_2(C_1)=\{U\pmod{p\Mat_{n-s}(\ZZ)}\mid(C_1,U)\ \text{symmetric pair}\}.$$
Similarly to Lemma \ref{lem construction A from D_1}, we have $\bar U=\bar U_1(I_{n-s}-p^{-1}C_1U_2\bar U_1)\pmod{C_1\Mat_n(\ZZ)}$. Then the remaining sum over $U$ is
\begin{align*}
\sum_{U_1\in\tilde X(p^{-1}C_1)}&\sum_{U_2\in\tilde X_2(C)}e((pC_1^{-1})\bar U_1(p^{-1}Q_3)+(pC_1^{-1})U_1(p^{-1}T_3))\\
	&=p^{(n-s)(n-s+1)/2}\cdot K_{n-s}(p^{-1}Q_3,p^{-1}T_3;p^{-1}C).
\end{align*}
This proves the first part of the first equation. The second part follows by Proposition \ref{pro sigma_1=0}. Note that the above computation is also valid if $s=0$. In that case, the sums over $W$ and $X$ are empty and we get
$$K_n(Q,T;C)=p^{n(n+1)/2}K_n(p^{-1}Q,p^{-1}T;p^{-1}C).$$

Let $k$ be the largest integer such that $C_k^{-1}Q$ and $C_k^{-1}T$ are integral and let $s$ be the largest index such that $\sigma_s\leq k$ ($s=0$ if $\sigma_1>k$). If $s=0$, by induction on the above, we have
$$K_n(Q,T;C)=p^{n(n+1)k/2}K_n(p^{-k}Q,p^{-k}T;p^{-k}C).$$
Therefore the second result of the proposition is true in that case, since $\sigma_{i,k}=k$ for all $i$. Suppose that $s\geq1$. Let $s_1$ be the largest index such that $\sigma_{s_1}=\sigma_1$. Applying the above equation with $k=\sigma_1-1$ and the first result, we get
\begin{align*}
K_n(Q,T;C)&=p^{n(n+1)(\sigma_1-1)/2}\abs{X(pI_{s_1})}p^{(n-s_1)(n+s_1+1)/2}K_{n-s}(p^{-\sigma_1}Q_3',p^{-\sigma_1}T_3';p^{-\sigma_1}C_3')\\
	&\ll p^{n(n+1)\sigma_1/2}K_{n-s_1}(p^{-\sigma_1}Q_3',p^{-\sigma_1}T_3',p^{-\sigma_1}C_3'),
\end{align*}
with $Q_3',T_3',C_3'$ the $n-s_1$ by $n-s_1$ bottom-right block of resp. $Q,T,C$. By induction on $n$, we have
\begin{align*}
K_n(Q,T;C)&\ll p^{n(n+1)\sigma_1/2}\prod_{i=s_1+1}^np^{(n-i+1)(\sigma_{i,k}-\sigma_1)}\cdot\abs{K_{n-s}(p^{-k}Q_3,p^{-k}T_3,p^{-k}C_3)},\\
	&=\prod_{i=1}^np^{(n-i+1)\sigma_{i,k}}\cdot\abs{K_{n-s}(p^{-k}Q_3,p^{-k}T_3,p^{-k}C_3)},
\end{align*}
since $\sigma_1=\sigma_{i,k}$ for $i\leq s_1$.
\end{proof}

\section{Matrix Gauss sums and a counting problem}\label{sec matrix Gauss sums}

In this section, we finalize the proof of Theorem \ref{thm bound for higher prime powers} by giving non-trivial bounds for exponential sums and a counting problem appearing in the last two sections. More precisely, we will consider the following elements of Proposition \ref{pro section block decomposition}:
\begin{enumerate}
\item The sum over $W$ is bounded trivially, except in some cases. More details are given in Remark (2) after Proposition \ref{pro Gauss sum over matrices}.
\item The sum over $U_1$ contains a quadratic equation. We bound the number of solutions in Proposition \ref{pro solutions T=UQU}.
\item The sum over $X$ is a Gauss sum over matrices. We bound it in Proposition \ref{pro Gauss sum over matrices}.
\item The sum over $U_{2,1}$ is a Gauss sum over coprime symmetric pairs. We bound it in Proposition \ref{pro symmetric sum over matrices}.
\end{enumerate}
Before that, we state two lemmas about cancellations of full matrix sums. We also consider multiple simple matrix equations in Lemma \ref{lem simpler matrix equations}.

From now, we suppose that $p\neq2$ is odd. In particular, half-integral matrices can be seen as integral matrices since 2 is invertible modulo $p$. We state Propositions \ref{pro Gauss sum over matrices}, \ref{pro symmetric sum over matrices} and \ref{pro solutions T=UQU} in a way that is true for all $p$, but only prove them for $p\neq2$. We consider the case $p=2$ at the end of the section.

\begin{lemma}\label{lem character sum over matrices}
Let $m,n$ be two positive integers. Let $C=\diag(p^{\sigma_1},\dots,p^{\sigma_m})$ and let $A$ be a $n$ by $m$ integral matrix. We have
$$\sum_{X\Mod{C\Mat_{m,n}(\ZZ)}}e(C^{-1}XA)=\delta_{A=0\Mod{\Mat_{n,m}(\ZZ)C}}\det(C)^n.$$
\end{lemma}

\begin{proof}
Let $X=(x_{ij})$ and $A=(a_{ij})$. We have
$$\tr(C^{-1}XA)=\sum_{i=1}^m\sum_{j=1}^np^{-\sigma_i}x_{ij}a_{ji}.$$
For fixed $i$ and $j$, the sum over $x_{ij}$ is a complete character sum. It is 0 unless
$$0=p^{-\sigma_i}a_{ji}=(AC^{-1})_{ji}\pmod 1$$
In conclusion, the full sum is 0 unless $AC^{-1}$ is integral. In the latter case, all the summands are 1 and the sum is equal to the size of the sum set.
\end{proof}

\begin{lemma}\label{lem symmetric character sum over matrices}
Let $p$ be an odd prime. Let $C=\diag(p^{\sigma_1},\dots,p^{\sigma_n})$ with $0\leq\sigma_1\leq\dots\leq\sigma_n$ and let $A\in\Mat_n(\ZZ)$. We have
$$\sum_{\substack{D\Mod{C\Mat_n(\ZZ)}\\(C,D)\ \text{sym. pair}}}e(C^{-1}DA)=\delta_{A+A^t=0\Mod{[C]}}\prod_{i=1}^np^{(n-i+1)\sigma_i}.$$
\end{lemma}

\begin{proof}
Let $X=(x_{ij})$ and $A=(a_{ij})$. We have
$$\tr(C^{-1}XA)=\sum_{\substack{i,j=1\\ i<j}}^np^{-\sigma_i}x_{ij}(a_{ji}+a_{ij})+\sum_{i=1}^np^{-\sigma_i}x_{ii}a_{ii}.$$
For fixed $i\leq j$, the sum over $x_{ij}$ is a complet character sum. For $i<j$, it is zero unless $0=p^{-\sigma_i}(a_{ij}+a_{ji})\pmod 1$ for all $1\leq i\leq j\leq n$. For $i=j$, it is zero unless $0=p^{-\sigma_i}a_{ii}\pmod1$. Since $p\neq2$ and $C$ has increasing prime powers on the diagonal, this is equivalent to $0=a_{ij}+a_{ji}\pmod{p^{\min\{\sigma_i,\sigma_j\}}}$ for all $1\leq i,j\leq n$.  That equation is the same as
$$0=A+A^t\pmod{[C]}.$$
In conclusion, the full sum is 0 unless the equation above is true. In the latter case, all the summands are 1 and the sum is equal to the size of the sum set.
\end{proof}

The following lemma is about various matrix equations. We prove non-trivial bounds for the number of solutions, but do not try to get the best possible result. Since $p$ is odd, it is equivalent to consider half-integral or integral matrices. We state the lemma for the former, since these matrices will appear in the applications.

\begin{lemma}\label{lem simpler matrix equations}
Let $m,n,k$ be positive integers and let $p$ be a prime number.
\begin{enumerate}
\item Let $T\in\Mat_{m,n}(\ZZ)$ and $Q\in\Mat_n(\RR)$ be a half-integral matrices with $(p,Q)=1$. The number of matrices $U\pmod{p^k\Mat_{m,n}(\ZZ)}$ satisfying the equation
$$T=UQ\pmod{p^k\Mat_{m,n}(\ZZ)}$$
is $O(p^{km(n-1)})$.

\item Let $T\in\Mat_{n,m}(\ZZ)$ and $Q\in\Mat_{n,m}(\ZZ)$ be a matrix with $(Q,p)=1$. The number of symmetric matrices $U\pmod{p^k\Mat_n(\ZZ)}$ satisfying the equation
$$T=UQ\pmod{p^k\Mat_n(\ZZ)}$$
is $O(p^{kn(n-1)/2})$.

\item Let $T\in\Mat_n(\RR)$ be a half-integral symmetric matrix and $D=\diag(d_1,\dots,d_m)$ be a diagonal matrix with $p\ndivides d_i$, $i=1,\dots,m$. Suppose that $m\geq3$ or that $m=1,2$ and $(p,T)=1$. Then the number of matrices $U\pmod{p^k\Mat_{n,m}(\ZZ)}$ satisfying the equation
$$T=UDU^t\pmod{p^k\Mat_n(\ZZ)}$$
is $O(p^{k(m-1)n})$. \emph{Remark:} a more precise result was proven by Carlitz \cite{Car54} for finite fields.

\item Let $T\in\Mat_n(\RR)$ be a half-integral symmetric matrix and $D=\diag(d_1,\dots,d_n)$ be a diagonal matrix with $p\ndivides d_i$, $i=1,\dots,n$. Suppose that $n\geq4$ or that $(p,T)=1$. The number of symmetric matrices $U\pmod{p^k\XXc_n(\ZZ)}$ satisfying the equation
$$T=UDU\pmod{p^k\Mat_n(\ZZ)}$$
is $O(p^{kn(n-1)/2})$.

\item Suppose that $p$ is odd. Let $T\in\Mat_n(\RR)$ be a half-integral symmetric matrix and $Q\in\Mat_{n,m}(\ZZ)$ with $(p,Q)=1$. The number of matrices $U\pmod{p^k\Mat_{n,m}(\ZZ)}$ satisfying the equation
$$T=QU^t+UQ^t\pmod{p^k\Mat_n(\ZZ)}$$
is $O(p^{km(n-1)})$. \emph{Remark:} a variant of this equation with symmetric $U$ is considered in Case 1 of the proof of Proposition \ref{pro symmetric sum over matrices}.
\end{enumerate}
\end{lemma}

\begin{proof}
We start with a preliminary claim. Let $v_p(x)$ be the minimum between the $p$-adic valuation of $x$ and $k$. Note that the number of $x\pmod{p^k}$ with valuation at least $v\in\RR_{\geq0}$ is $O(p^{k-v})$.\\
\emph{Claim}: the number of solutions of the equation $x^2=a\pmod{p^k}$ is bounded by $O(p^{v_p(a)/2})$.\\
\emph{Proof of Claim}: note that if $a=0\pmod{p^k}$, then the solutions of the equation are the $x$ with $v_p(x)\geq k/2$. Suppose that $a\neq0\pmod{p^k}$. Write $a=p^rb$ and $x=p^sy$ with $r=v_p(a)$ and $s=v_p(x)$. Clearly, we must have $r=2s$. Then $b=y^2\pmod{p^{k-2s}}$. The number of solution of this equation with $p\ndivides b$ is bounded (it is at most 4 for $p=2$ and at most 2 otherwise). Let $y_0$ be such a solution (if it exists). Then $x=p^s(y_0+p^{k-2s}y_1)$ for some $y_1$. The different values possible for $y_1$ are $0,\dots,p^s-1$.

Now, we prove the statements (1)--(5).
\begin{enumerate}
\item Since $(p,Q)=1$, there is $1\leq k_0,j_0\leq n$ with $p\ndivides q_{k_0j_0}$. Let $1\leq i\leq m$. We have
$$t_{ij_0}=\sum_{k=1}^nu_{ik}q_{kj_0}\pmod{p^k}.$$
Fix $u_{ik}\pmod{p^k}$ for $1\leq i\leq m$ and $k\neq k_0$. Then $u_{ik_0}$ is given by the equation above since $q_{k_0,j_0}$ is invertible. In total, we have at most $O(p^{km(n-1)})$ solutions.

\item Since $(p,Q)=1$, there is $1\leq k_0\leq n$, $1\leq j_0\leq m$ with $p\ndivides q_{k_0j_0}$. By multiplying on the right by a permutation matrix, we can suppose without loss of generality that $j_0=m$. Let $1\leq i\leq n$. We have
$$t_{im}=\sum_{k=1}^nu_{ik}q_{km}\pmod{p^k}.$$
Recall that $U$ is symmetric so $u_{ik}=u_{ki}$. Fix $u_{ik}\pmod{p^k}$ for $i,k\neq k_0$. Then $u_{ik_0}\pmod{p^k}$ is fixed by the above equation for $i\neq k_0$. Once all the values except $u_{k_0k_0}$ are fixed, the last coordinate is fixed by considering the above equation with $i=k_0$. In total, we have at most $O(p^{kn(n-1)/2})$ solutions.

\item\emph{Case $m=1$}: suppose that $(p,T)=1$. Let $1\leq i,j\leq n$. Then
$$t_{ij}=d_1u_{i1}u_{j1}\pmod{p^k}.$$
There is $i_0,j_0$ such that $p\ndivides t_{i_0j_0}$. Then $p\ndivides u_{i_01},u_{j_01}$. We deduce that $p\ndivides t_{i_0i_0}=d_1u_{i_01}^2$. By the claim, there are $O(1)$ for $u_{i_01}$. Then for all $j\neq i_0$, we can fix $u_{j1}=\bar d_1\bar u_{i_01}t_{i_0j}\pmod{p^k}$. So there are finitely many solutions in that case.\\
\emph{Case $m=2$}: first, we give a bound for a general $T$. Let $1\leq i\leq n$. We have
\begin{align}\label{eq SME1 t_ii}
t_{ii}=d_1u_{i1}^2+d_2u_{i2}^2\pmod{p^k}.
\end{align}
For a fixed $u_{i1}$, we have $O(p^{v/2})$ solutions for $u_{i2}$ with $v=v_p(t_{ii}-d_1u_{i1}^2)$ by the claim. We get the additional equation
\begin{align}\label{eq SME1 additional equation}
t_{ii}=d_1u_{i1}^2\pmod{p^v}
\end{align}
The number of solutions for $u_{i1}\pmod{p^v}$ is bounded by $O(p^{v/2})$ and the number of ways to lift a solution modulo $p^k$ is $O(p^{k-v})$. Therefore, the number of solutions for the pair $(u_{i1},u_{i2})$ is
$$O\left(\sum_{v=0}^kp^{k-v+v/2+v/2}\right)=O(kp^k).$$
Now suppose that $(p,T)=1$. There is $i_0,j_0$ such that $p\ndivides t_{i_0j_0}$. Suppose that $i_0=j_0$. Consider Equation \eqref{eq SME1 t_ii} for $i=i_0$ and the computation below it. In Equation \eqref{eq SME1 additional equation}, the number of $u_{i_01}\pmod{p^v}$ is $O(1)$. Thus we get
$$O\left(\sum_{v=0}^kp^{k-v+v/2}\right)=O(p^k)$$
solutions for the pair $(u_{i_01},u_{i_02})$ in that case. Suppose that $p\ndivides u_{i_01}$. Then consider $j\neq i_0$ and
\begin{align}\label{eq SME1 t_ij}
t_{i_0j}=d_1u_{i_01}u_{j1}+d_2u_{i_02}u_{j2}\pmod{p^k}.
\end{align}
Once $u_{j2}$ is fixed, so is $u_{j1}$ since $u_{i_01}$ is invertible. If $p\divides u_{i_01}$, then $p\ndivides u_{i_02}$ and the proof goes the same way with $u_{i_02}$ instead of $u_{i_01}$. We get $O(p^{nk})$ solutions for $U$ in that case.

Suppose that $p\divides t_{ii}$ for all $i$. Let $i_0\neq j_0$ with $p\ndivides t_{i_0j_0}$. Consider the equation
$$t_{i_0j_0}=d_1u_{i_01}u_{j_01}+d_2u_{i_02}u_{j_02}\pmod{p^k}.$$
Suppose that $p\ndivides u_{i_01},u_{j_01}$. Otherwise the same proof works with $p\ndivides u_{i_02},u_{j_02}$. Fix these two coordinates arbitrarily. There are $O(p^{2k})$ possible choices. Consider
$$t_{i_0i_0}=d_1u_{i_01}^2+d_2u_{i_02}^2\pmod{p^k}.$$
Then there are $O(1)$ choices for $u_{i_02}$ by the claim since $p\divides t_{i_0i_0}$. Fix $u_{j_02}$ in the same way. Then consider $j\neq i_0,j_0$ and solve Equation \eqref{eq SME1 t_ij} as in the case $i_0=j_0$. In total, we get $O(p^{kn})$ solutions for $U$ in both cases.\\
\emph{Case $m=3$}: let $1\leq i\leq n$. Consider
$$t_{ii}=d_1u_{i1}^2+d_2u_{i2}^2+d_3u_{i3}^2\pmod{p^k}.$$
For fixed $u_{i1},u_{i2}$, we have $O(p^{v/2})$ solutions for $u_{i3}$ with $v=v_p(t_{ii}-d_1u_{i1}^2-d_2u_{i2}^t)$ by the claim. We get the additional equation
$$t_{ii}=d_1u_{i1}^2+d_2u_{i2}^2\pmod{p^v}$$
This has $O((v+1)p^v)$ solutions by the case $m=2$. The number of ways to lift them modulo $p^k$ is $O(p^{2(k-v)})$. In total, the number of solutions for $(u_{i1},u_{i2},u_{i3})$ is bounded by
$$\ll\sum_{v=0}^k(v+1)p^{2(k-v)+v+v/2}\ll\sum_{v=0}^k(v+1)p^{2k-v/2}\ll p^{2k}.$$
We conclude by summing over $i$. The number of solutions for $U$ is $O(p^{2kn})$.\\
\emph{Case $m\geq4$}: for $1\leq i\leq n$, we have
$$t_{ii}=\sum_{j=1}^md_ju_{ij}^2\pmod{p^k}.$$
By induction on $m$, suppose that the above equation has $O(p^{k'(m'-1)})$ solutions for $m'=m-1$ and all $k'\in\NN$. We proved this above for $m'=3$. Let $r=t_{ii}-\sum_{j=1}^{m-1}d_ju_{ij}^2$. The number of solutions for $u_{in}$ is bounded by $p^{v/2}$ with $v=v_p(r)$. Moreover the equation $r=0\pmod{p^v}$ has $O(p^{v(m-2)})$ solutions for $(u_{i1},\dots,u_{i,m-1})$ by induction. These can be lift modulo $p^k$ in at most $p^{(k-v)(m-1)}$ ways. In total, the number of solutions for $(u_{i1},\dots,u_{im})$ is bounded by
$$\ll\sum_{v=0}^kp^{k(m-1)-v+v/2}\ll p^{k(m-1)}.$$
We conclude by summing over $i$.

\begin{remark}
In particular, we showed that the equation
$$t=\sum_{j=1}^md_jx_j^2\pmod{p^k}$$
has $O(p^{k(m-1)})$ solutions for $m\geq3$ and the same is true for $m=1,2$ if $p\ndivides t$.
\end{remark}

\item\emph{Case $n=1$}: the equation is the same as in $(1)$.\\
\emph{Case $n=2$}: suppose that $(p,T)=1$. we have the equations
\begin{align*}
t_{11}&=d_1u_{11}^2+d_2u_{12}^2 &\pmod{p^k},\\
t_{12}&=u_{12}(d_1u_{11}+d_2u_{22}) &\pmod{p^k},\\
t_{22}&=d_1u_{12}^2+d_2u_{22}^2 &\pmod{p^k}.
\end{align*}
Suppose that $p\ndivides t_{11},t_{22}$. Then there are $O(p^k)$ way to solve the first equation by (1). If $p\ndivides u_{12}$, then $u_{22}$ is fixed by the second equation. If $p\divides u_{12}$, then there are $O(1)$ solutions for $u_{22}$ in the last equation by the claim.

Suppose that $p\divides t_{11},t_{22}$ and $p\ndivides t_{12}$. Then $p\ndivides u_{12}$ by the second equation. Once $u_{12}$ is fixed, there are $O(1)$ choices for $u_{11}$ and $u_{22}$ by the claim using the first resp. the last equation.

Suppose that $p\ndivides t_{11}$ and that $p\divides t_{22}$. Combining the first and the last equation, we get
$$d_1t_{11}-d_2t_{22}=d_1d_2(u_{11}^2-u_{22}^2)\pmod{p^k}.$$
By (1), Case $m=2$, we have $O(p^k)$ solutions for the pair $(u_{11},u_{22})$. Then if $p\divides u_{22}$, we have $p\ndivides u_{11}$ and $u_{12}$ is fixed by the second equation. If $p\ndivides u_{22}$, then $u_{12}$ is fixed by the last equation using the claim. If $p\divides t_{11}$ and $p\ndivides t_{22}$, the same proof works if we exchange the roles of $u_{11}$ and $u_{22}$. In any case, we got $O(p^k)$ solutions for $U$.\\
\emph{Case $n=3$}: let $P$ be a permutation matrix. Consider the equivalent equation
$$PTP^t=VP^{-t}DP^{-1}V\pmod{p^k}$$
with $V=PUP^t$. We can make the change of variable $U\mapsto V$ and choose $P$ such that $p\ndivides (PTP^t)_{i_0j_0}$ for fixed $i_0,j_0$ with $i_0,j_0\geq2$. Without loss of generality, we suppose that this holds for $T$. Consider the equation
$$t_{11}=d_1u_{11}^2+d_2u_{12}^2+d_3u_{13}^2\pmod{p^k}.$$
We have $O(p^{2k})$ solutions for $t_{11}$ as seen in (1). Consider the bottom-right block of size 2 by 2 of the equation. We have 3 equations for $u_{22},u_{23}$ and $u_{33}$. This corresponds to the case $n=2$ with $T$ replaced by some combination of $T$ and $u_{12},u_{13}$. If $p\divides u_{12},u_{13}$, then we get back the case $n=2$ and have $O(p^k)$ solutions for $(u_{22},u_{23},u_{33})$. Otherwise, consider the equations
\begin{align*}
t_{12}&=d_1u_{11}u_{12}+d_2u_{12}u_{22}+d_3u_{13}u_{23}\pmod{p^k},\\
t_{13}&=d_1u_{11}u_{13}+d_2u_{12}u_{23}+d_3u_{13}u_{33}\pmod{p^k}.
\end{align*}
If $p\ndivides u_{12}$, fix $u_{33}$. Then $u_{23}$ is fixed by the second equation. Once $u_{23}$ is fixed, $u_{22}$ is fixed by the first equation. If $p\ndivides u_{13}$, fix $u_{22}$. Then $u_{23}$ is fixed by the first equation. Once $u_{23}$ is fixed, $u_{33}$ is fixed by the second equation. In any case, we get $O(p^{4k})$ solutions for $U$.\\
\emph{Case $n=4$}: let $v_1=v_p((u_{13},u_{14}))$ and $i_0$ such that $v(u_{1i_0})=v_1$. Consider the equation
$$t_{11}-\sum_{i\neq i_0}d_iu_{1i}^2=d_{i_0}u_{1i_0}^2\pmod{p^k}.$$
The right-hand side has valuation $2v_1$ and so has the left-hand side. Therefore, the number of solutions for $u_{1i_0}^2$ once the rest is fixed is $O(p^{v_1})$ by the claim. For $i_0\neq i=3,4$, we fix $u_{1i}$. Since $v_p(u_{1i})\geq v_1$, we have $O(p^{k-v_1})$ possibilities. We do something similar for the second row. Let $v_2=v_p((u_{23},u_{33}))$ and $i_0$ the coordinate such that $v_p(u_{2i_0})=v_2$. Then the number of solutions for $u_{2i_0}$ once the rest is fixed is $O(p^{v_2})$. For $i_0\neq i=3,4$, we have $O(p^{k-v_2})$ possibilities to fix $u_{2i}$. In total, we have $O(p^{2k})$ solutions for $(u_{13},u_{14},u_{23},u_{24})$ once $u_{11},u_{12},u_{22}$ are fixed.

We are left with the equations
\begin{align*}
t_{11}&=d_1u_{11}^2+d_2u_{12}^2\pmod{p^{2v_1}},\\
t_{22}&=d_1u_{12}^2+d_2u_{22}^2\pmod{p^{2v_2}}.
\end{align*}
If $v_1<v_2$, we consider the first equation. By (1), Case $m=2$, we have $O((v_1+1)p^{2v_1})$ solutions for the pair $(u_{11},u_{12})$. We have $O(p^{2(k-2v_1)})$ ways to lift them modulo $p^k$. Then in the second equation, we have $O(p^{v_2})$ solutions for $u_{22}$ by the claim and we can lift them in $O(p^{k-2v_2})$ ways modulo $p^k$. In total, we get $O((v_1+1)p^{3k-2v_1-v_2})$ solutions in that case.

If $v_1\geq v_2$, we exchange the roles of $v_1$ and $v_2$. That is we consider the second equation. By (1), Case $m=2$, we have $O((v_2+1)p^{2k-2v_2})$ solutions for the pair $(u_{12},u_{22})\pmod{p^k}$. Then in the first equation, we have $O(p^{k-v_1})$ solutions for $u_{11}\pmod{p^k}$ by the claim. In total, we get $O((v_2+1)p^{3k-2v_2-v_1})$ solutions in that case.

Finally, let $v=\min\{v_1,v_2\}$. We write $T=(T_{ij})$, $U=(U_{ij})$, $D=\diag(D_1,D_2)$ in blocks of size 2 by 2. We have
$$T_{22}-U_{11}D_1U_{12}=U_{12}D_2U_{22}\pmod{p^k\Mat_2(\ZZ)}.$$
We fixed $U_{11}$ and $U_{12}$. Both sides are divisible by $p^v$ and $(p,p^{-v}U_{12})=1$. By Lemma \ref{lem simpler matrix equations} (2), we get $O(p^{k-v})$ solutions for $U_{22}\pmod{p^{k-v}\Mat_2(\ZZ)}$. We have $O(p^{3v})$ ways to lift the solutions modulo $p^k\Mat_2(\ZZ)$. In total, the number of solutions for $U$ is bounded by
\begin{align*}
	&\ll\sum_{v_1=0}^k\left(\sum_{v_2=0}^{v_1}(v_2+1)p^{2k}p^{3k-2v_2-v_1}p^{k+2v_2}+\sum_{v_2=v_1+1}^k(v_1+1)p^{2k}p^{3k-2v_1-v_2}p^{k+2v_1}\right)\\
	&\ll\sum_{v_1=0}^kp^{6k}((v_1+1)^2p^{-v_1}+(v_1+1)p^{-v_1})\\
	&\ll p^{6k}.
\end{align*}
\emph{Case $n\geq5$}: for $1\leq i\leq n-4$, we have
$$t_{ii}-\sum_{j=1}^{i-1}d_iu_{ji}^2=\sum_{j=i}^nd_iu_{ij}^2\pmod{p^k}.$$
Consider $i$ in increasing order. By (1), we have $O(p^{k(n-i)})$ solutions for $(u_{ii},\dots,u_{in})$ if $i\leq n-4$. Finally, once $u_{ij}$ is fixed for $1\leq i\leq j\leq n-4$, we get a 4 by 4 symmetric matrix equation that corresponds to the case $n=4$. In total, we get
$$O\left(\sum_{i=1}^{n-4}p^{k(n-i)}\cdot p^{6k}\right)=O(p^{kn(n-1)/2})$$
solutions for $U$.

\item Since $(p,Q)=1$, there is $1\leq j_0\leq n$, $1\leq k_0\leq m$ with $p\ndivides q_{j_0k_0}$. For $1\leq i\leq n$, we have
$$t_{ij_0}=\sum_{k=1}^n(q_{ik}u_{j_0k}+q_{j_0k}u_{ik})\pmod{p^k}.$$
Fix $u_{ik}$ for all $i$ and $k\neq k_0$. Consider the above equation for $i=j_0$. We get
$$2q_{j_0k_0}u_{j_0k_0}=t_{j_0j_0}-2\sum_{k\neq k_0}q_{j_0k}u_{j_0k}\pmod{p^k}.$$
Since $p\neq2$, this fixes $u_{j_0k_0}$. Now consider the above equation for $i\neq j_0$. We get
$$q_{j_0k_0}u_{ik_0}=t_{ij_0}-q_{ik_0}u_{j_0k_0}-\sum_{k\neq k_0}(q_{ik}u_{j_0k}+q_{j_0k}u_{ik})\pmod{p^k}.$$
Everything on the right is fixed so this fixes $u_{ik_0}$ for $i\neq j_0$. In total, we have at most $O(p^{k(m-1)n})$ solutions.
\end{enumerate}
\end{proof}

\begin{proposition}\label{pro Gauss sum over matrices}
Let $p$ be an odd prime. Let $A\in\Mat_{n-s,s}(\RR)$ with half-integral coefficients, $B_1\in\XXc_{n-s}(\RR)$ a half-integral symmetric matrix and $B_2\in\XXc_s(\ZZ)$ with $p\ndivides\det(B_2)$. Consider the sum
$$G(A,B_1,B_2;p):=\sum_{X\Mod{p\Mat_{s,n-s}(\ZZ)}}e(2p^{-1}XA+p^{-1}XB_1X^tB_2).$$
If $(p,2B_1)=1$, then
$$\abs{G(A,B_1,B_2;p)}\ll p^{s(n-s-1/2)}.$$
Also if $(p,2B_1)\neq1$, then
$$\abs{G(A,B_1,B_2;p)}\ll\delta_{2A=0\Mod{p\Mat_{n-s,s}(\ZZ)}}p^{s(n-s)}.$$
Finally if $p\ndivides\det(B_1)$, then
$$\abs{G(A,B_1,B_2;p)}\leq p^{s(n-s)/2}.$$
\end{proposition}

\begin{remark}\ 
\begin{enumerate}
\item A precise computation of a similar Gauss sum was done by Walling in \cite{Wal00}.
\item In Proposition \ref{pro section block decomposition}, the sum over $X\pmod{p\Mat_{s,n-s}(\ZZ)}$ is $G(T_2^t-\bar U_1Q_2^t\bar W,T_3,\bar W)$. In particular, when we are in the case where $2A=0\pmod{p\Mat_{n-s,s}}$, we have
$$2T_2U_1=2\bar WQ_2\pmod{p\Mat_{s,n-s}(\ZZ)}.$$
By Lemma \ref{lem simpler matrix equations} (2), we get $O(p^{s(s-1)/2})$ possibilities for $\bar W$ if $(p,2Q_2)=1$. Since summing over $W$ is equivalent to summing over $\bar W$, this means that the combined sum over $W$ and $X$ is bounded by
\begin{align*}
&\ll\min\{p^{s(s+1)/2}\cdot p^{s(n-s-1/2)}(p,2T_3)^{s/2},p^{s(s-1)/2}(p,2Q_2)^s\cdot p^{s(n-s)}\}\\
&\ll p^{s(n-s/2)}(p,2Q_2,2T_3)^{s/2}.
\end{align*}
Finally, by the remark after Proposition \ref{pro solutions T=UQU}, we can replace $T_3$ by $Q_3$. We will use this bound at the end of this section when proving Theorem \ref{thm bound for higher prime powers}.
\end{enumerate}
\end{remark}

\begin{proof}
We compute the square of the absolute value of the sum:
\begin{align*}
|G(A,B_1,&B_2;p)|^2=\sum_{X_1,X_2\Mod{p\Mat_{s,n-s}(\ZZ)}}e(2p^{-1}(X_1-X_2)A+p^{-1}X_1B_1X_1^tB_2-p^{-1}X_2B_1X_2^tB_2)\\
	&=\sum_{X_1,X_2\Mod p}e(2p^{-1}(X_1-X_2)A+p^{-1}(X_1+X_2)B_1(X_1^t-X_2^t)B_2\\
	&\quad+p^{-1}X_1B_1X_2^tB_2-p^{-1}X_2B_1X_1^tB_2).
\intertext{We replace $X_2$ by $X=X_1-X_2$.}
	&=\sum_{X_1,X\Mod p}e(2p^{-1}XA+p^{-1}(2X_1-X)B_1X^tB_2-p^{-1}X_1B_1X^tB_2+p^{-1}XB_1X_1^tB_2).
\intertext{The sum over $X_1$ is now linear.}
	&=\sum_{X\Mod p}e(2p^{-1}XA-p^{-1}XB_1X^tB_2)\sum_{X_1\Mod p}e(p^{-1}X_1B_1X^tB_2+p^{-1}XB_1X_1^tB_2).
\end{align*}
Recall that $B_1$ and $B_2$ are symmetric. We rearrange the sum over $X_1$ and apply Lemma \ref{lem character sum over matrices}:
\begin{align*}
\sum_{X_1\Mod p}e(p^{-1}X_1B_1X^tB_2+p^{-1}XB_1X_1^tB_2)&=\sum_{X_1\Mod p}e(2p^{-1}X_1B_1X^tB_2)\\
	&=\delta_{2B_1X^tB_2=0\Mod{p\Mat_{n-s,s}(\ZZ)}}p^{s(n-s)}.
\end{align*}
Since $p\ndivides\det(B_2)$, we have
$$2B_1X^tB_2=0\pmod{p\Mat_{n-s,s}(\ZZ)}\Leftrightarrow2B_1X^t=0\pmod{p\Mat_{n-s,s}(\ZZ)}.$$
If $p\ndivides\det(B_1)$, clearly the only solution is $X=0\pmod{p\Mat_{s,n-s}(\ZZ)}$. If $(p,B_1)=1$, by Lemma \ref{lem simpler matrix equations} (1), there are $O(p^{s(n-s-1)})$ possible values of $X\pmod p$. Then we have
\begin{align*}
\abs{G(A,B_1,B_2;p)}^2&=p^{(n-s)s}\sum_{X\Mod{p\Mat_{s,n-s}(\ZZ)}}e(2p^{-1}XA-p^{-1}XB_1X^tB_2)\delta_{X\in p\Mat_{s,n-s}(\ZZ)}\\
	&\ll p^{s(n-s)}\cdot p^{s(n-s-1)}\\
	&=p^{2s(n-s-1/2)}.
\end{align*}
Finally if $(p,B_1)\neq1$, then the original sum is
$$G(A,0,B_2;p)=\sum_{X\Mod{p\Mat_{s,n-s}(\ZZ)}}e(2p^{-1}XA)=p^{s(n-s)}\delta_{A=0\Mod{p\Mat_{n-s,s}}}.$$
\end{proof}

\begin{proposition}\label{pro symmetric sum over matrices}
Let $p$ be an odd prime. Let $A,B\in\Mat_n(\RR)$ be half-integral symmetric matrix and $W\in\Mat_n(\ZZ)$ with $(C,W)$ a coprime symmetric pair. We consider the sum
$$H(A,B,W;C\tilde C^{-2}):=\sum_{\substack{U\Mod{C\tilde C^{-2}\Mat_n(\ZZ)}\\(C\tilde C^{-1},U)\ \text{sym. pair}}}e(C^{-1}\tilde C^2UA+C^{-1}\tilde CUW\tilde CUWBW^t)$$
\begin{enumerate}
\item If $C=p^kI_n$ is scalar, $k\geq2$, and $k$ is odd, then
$$\abs{H(A,B,W;C\tilde C^{-2})}\ll p^{n^2/2}(p,2B)^{n/2}.$$
If $k$ is even, the sum is 1 (there is only one matrix $U$ in the sum).
\item In general, we have
$$\abs{H(A,B,W;C\tilde C^{-2})}\ll\prod_{i=1}^np^{(n-i+1/2)(\sigma_i-2\mu_i)}(p,2B_i')^{(\sigma_i-2\mu_i)/2}.$$
Here $B_i'$ is the bottom-right block of $B$ of size $s$ by $s$, where $s$ is the smallest integer with $\sigma_s=\sigma_i$.
\end{enumerate}

\end{proposition}

\begin{remark}\ 
In Proposition \ref{pro section block decomposition}, the sum over $U_{2,1}\in\tilde X_{2,1}(C_1)$ is $H(T_3-\bar U_1Q_3\bar U_1^t,Q_3,\bar U_1;C_1\tilde C_1^{-2})$.
\end{remark}

\begin{proof}
We compute the square of the absolute value of the sum:
\begin{align*}
|H(A,B,W;&C\tilde C^{-2})|^2\\
	&=\sum_{U_1,U_2}e(C^{-1}\tilde C^2(U_1-U_2)A+C^{-1}\tilde CU_1W\tilde CU_1WBW^t-C^{-1}\tilde CU_2W\tilde CU_2WBW^t)\\
	&=\sum_{U_1,U_2}e(C^{-1}\tilde C^2(U_1-U_2)A+C^{-1}\tilde C(U_1+U_2)W\tilde C(U_1-U_2)WBW^t\\
	&\quad+C^{-1}\tilde CU_1W\tilde CU_2WBW^t-C^{-1}\tilde CU_2W\tilde CU_1WBW^t).
\intertext{We replace $U_2$ by $U=U_1-U_2$.}
	&=\sum_{U_1,U}e(C^{-1}\tilde C^2UA+C^{-1}\tilde C(U_1-U)W\tilde CUWBW^t+C^{-1}\tilde CUW\tilde CU_1WBW^t).
\intertext{The sum over $U_1$ is now linear.}
	&=\sum_Ue(C^{-1}\tilde C^2UA-C^{-1}\tilde CUW\tilde CUWBW^t)\\
	&\quad\cdot\sum_{U_1}e(C^{-1}\tilde CU_1W\tilde CUWBW^t+C^{-1}\tilde CUW\tilde CU_1WBW^t).
\end{align*}
Since $B$ is symmetric and $(C\tilde C^{-1},U)$, $(C\tilde C^{-1},U_1)$ and $(C,W)$ are symmetric pairs, the inner sum is equal to
$$\sum_{U_1}e(WBW^tU^t\tilde CW^tU_1^t\tilde CC^{-1}+C^{-1}\tilde CUW\tilde CU_1WBW^t)=\sum_{U_1}e(2U^tC^{-1}\tilde CW\tilde CU_1WBW^t).$$
Let $V=\tilde C^{-1}W\tilde CU_1W$. Then $(C\tilde C^{-1},V)$ is a symmetric pair since
$$C^{-1}\tilde CV=CW\tilde CU_1W=W^tU_1^t\tilde CW^tC^{-1}=VC^{-1}\tilde C.$$
Moreover let $U_1-U_1'\in C\tilde C^{-2}\Mat_n(\ZZ)$ and $V-V':=\tilde C^{-1}W\tilde C(U_1-U_1')W$. By Lemma \ref{lem integrality of G-1HG}, we have
$$V-V'\in(\tilde C^{-1}W\tilde C)C\tilde C^{-2}\Mat_n(\ZZ)=C\tilde C^{-2}(C^{-1}\tilde CWC\tilde C^{-1})\Mat_n(\ZZ)=C\tilde C^{-2}\Mat_n(\ZZ).$$
So $U_1\mapsto V$ is a valid change of variable. We get
$$\sum_{U_1}e(2U^tC^{-1}\tilde CW\tilde CU_1WBW^t)=\sum_Ve(2C^{-1}\tilde C^2VBW^tU^t).$$
Let $R=BW^t$ and $M=(m_{ij})=BW^tU^t$. We showed that
\begin{align}\label{eq symmetric Gauss sum squared}
\abs{H(A,B,W;C\tilde C^{-2})}^2=\sum_Ue(C^{-1}\tilde C^2UA-C^{-1}\tilde CUW\tilde CUWBW^t)\sum_Ve(2C^{-1}\tilde C^2VM).
\end{align}
Our goal now is to bound the number of $U$. The innermost summand gives
\begin{align*}
\tr(2C^{-1}\tilde C^2VM)&=2\sum_{\substack{i,j=1\\i<j}}^n(p^{2\mu_i-\sigma_i}m_{ji}+p^{2\mu_j-\sigma_j}p^{(\sigma_j-\mu_j)-(\sigma_i-\mu_i)}m_{ij})v_{ij}+\sum_{i=1}^np^{2\mu_i-\sigma_i}m_{ii}v_{ii}\\
	&=2\sum_{\substack{i,j=1\\i<j}}^np^{2\mu_i-\sigma_i}(m_{ji}+p^{\mu_j-\mu_i}m_{ij})v_{ij}+\sum_{i=1}^np^{2\mu_i-\sigma_i}m_{ii}v_{ii}.
\end{align*}
For fixed $i$ and $j$, we sum over $v_{ij}\pmod{p^{\sigma_i-2\mu_i}}$. This is a complete character sums and it cancels unless the coefficient in front is $0\pmod{p^{\sigma_i-2\mu_i}}$. In the latter case, it is equal to the number of $V\pmod{C\tilde C^{-2}\Mat_n(\ZZ)}$, which is $O(\prod_{i=1}^np^{(n-i+1)(\sigma_i-2\mu_i)})$. Assuming that $p$ is odd, we get in the former case
$$p^{\mu_j-\mu_i}m_{ij}+m_{ji}=0\pmod{p^{\sigma_i-2\mu_i}}$$
for $1\leq i\leq j\leq n$. This is equivalent to
\begin{align}\label{eq RU condition}
\begin{pmatrix}0&\cdots&0\\&\ddots&\vdots\\&&0\end{pmatrix}=\tilde C^{-1}RU^t\tilde C+UR^t\pmod{C\tilde C^{-2}\Mat_n(\ZZ)}
\end{align}
where we do not consider the equations given by coefficients under the diagonal.

\ \\
\emph{Claim}: the number of $U$ satisfying the above equation is
$$O\left(\prod_{i=1}^np^{(n-i)(\sigma_i-2\mu_i)}(p,R_i')^{\sigma_i-2\mu_i}\right).$$
Moreover $(p,R_i')^{\sigma_i-2\mu_i}=(p,B_i')^{\sigma_i-2\mu_i}$. Here $R_i'$ is the bottom-right block of $R$ of size $s$ by $s$, where $s$ is the smallest integer with $\sigma_s=\sigma_i$ and $B_i'$ is defined similarly.

\ \\
We split the proof of the claim into three cases, depending on the shape of $C$.\\
\emph{Case 1}: $C=p^kI_n$. In that case, $U$ is symmetric. If $k$ is even, there is nothing to prove. Suppose $k$ odd. In that case, Equation \eqref{eq RU condition} becomes
\begin{align}\label{eq RU symmetric}
(RU^t+UR^t)_{ij}=\sum_{k=1}^n(r_{ik}u_{jk}+r_{jk}u_{ik})=0\pmod p,\quad i\leq j
\end{align}
with $u_{jk}=u_{kj}$ since $U$ is symmetric. Since the equation is symmetric in $i$ and $j$, it is actually valid for any coordinate. If $(p,R)\neq1$, the equation is trivial. Otherwise the equation is similar to Lemma \ref{lem simpler matrix equations} (5), but we have to be careful with the additional symmetry.\\
\emph{Case 1.1}: $p\ndivides r_{j_0k_0}$ for some $j_0\neq k_0$. Fix $u_{jk}\pmod p$ for all $1\leq j\leq k$ except for $k=k_0$ or $j=k_0$. Equation \eqref{eq RU symmetric} for $i=j=j_0$ is
$$2r_{j_0k_0}u_{j_0k_0}=-2\sum_{k\neq k_0}r_{j_0k}u_{j_0k}\pmod p.$$
This fixes $u_{j_0k_0}$. Equation \eqref{eq RU symmetric} for $i\neq j=j_0$ is
$$r_{j_0k_0}u_{ik_0}=r_{ik_0}u_{j_0k_0}+\sum_{k\neq k_0}(r_{ik}u_{j_0k}+r_{j_0k}u_{ik})\pmod p.$$
If $i\neq k_0$, everything on the right-hand side of the equation is fixed and we get $u_{ik_0}$. Finally, consider the above equation for $i=k_0$. Now the right-hand side is fixed and we get $u_{k_0k_0}$. We fixed $n$ coordinates of $V$ from the others, meaning that we have at most
$$O(p^{n(n-1)/2})$$
solutions for $U$.\\
\emph{Case 1.2}: $p\divides r_{jk}$ for all $j\neq k$. Then Equation \eqref{eq RU symmetric} is
$$(r_{ii}+r_{jj})u_{ij}=0\pmod p.$$
This fixes $u_{ij}$ for all $i\leq j$ with $p\ndivides r_{ii}+r_{jj}$. Let $j_0$ be such that $p\ndivides r_{j_0j_0}$. For all $1\leq i\leq n$, we do the following: if $p\ndivides r_{ii}+r_{j_0j_0}$, we fix $u_{ij_0}$ with the above equation. If $p\divides r_{ii}+r_{j_0j_0}$, then clearly $p\ndivides r_{ii}$. We fix $u_{ii}$ with the above equation. We fixed $n$ coordinates of $V$, meaning that we have at most
$$O(p^{n(n-1)/2})$$
possible solutions for $U$.

We got the same bound for Case 1.1 and Case 1.2. Adding the case $(p,R)\neq1$, we get
$$O(p^{n(n-1)/2}(p,R)^n)=O\left(\prod_{i=1}^np^{n-i}(p,R)\right).$$
solution for $U$. Note that $R=BW^t$ and $p\ndivides\det(W)$. So $(p,R)=(p,B)$.

\ \\
\emph{Case 2}: $C$ not scalar. We prove the claim by induction on $n$.\\
\emph{Case 2.1}: $\sigma_1$ is even. Let $\sigma=\sigma_1$ and $\mu=\mu_1$. We write $C=\diag(p^\sigma I_r,C_1)$ with all the prime powers in $C_1$ larger than $p^\sigma$. Then $\tilde C=\diag(p^\mu I_r,\tilde C_1)$ and $C\tilde C^{-1}=\diag(p^{\sigma-\mu}I_r,C_1\tilde C_1^{-1})$. Note that $p^\mu$ can be a prime power of $\tilde C_1$. We write
$$R=\pmatrixtwo{R_1}{R_2}{R_3}{R_4},\quad U=\pmatrixtwo{U_1}{U_2}{p^{\mu-\sigma}C_1\tilde C_1^{-1}U_2^t}{U_4}$$
with $R_1$ and $U_1$ blocks of size $r$ by $r$. Note that $p^{\mu-\sigma}C_1\tilde C_1^{-1}U_2^t=0\pmod{p\Mat_{n-r,r}(\ZZ)}$ and $U_1$ is symmetric. Equation \eqref{eq RU condition} becomes
\begin{align*}
\begin{pmatrix}0&\cdots&0\\&\ddots&\vdots\\&&0\end{pmatrix}&=\tilde C^{-1}RU^t\tilde C+UR^t\pmod{C\tilde C^{-2}\Mat_n(\ZZ)}\\
	&=\pmatrixtwo{p^{-\mu}}{}{}{\tilde C_1^{-1}}\pmatrixtwo{R_1}{R_2}{R_3}{R_4}\pmatrixtwo{U_1}{}{U_2^t}{U_4^t}\pmatrixtwo{p^\mu}{}{}{\tilde C_1}+\pmatrixtwo{U_1}{U_2}{}{U_4}\pmatrixtwo{R_1^t}{R_3^t}{R_2^t}{R_4^t}\\
	&=\pmatrixtwo\ast{p^{-\mu}R_2U_4^t\tilde C_1+U_1R_3^t+U_2R_4^t}\ast{\tilde C_1^{-1}R_4U_4^t\tilde C_1+U_4R_4^t}\pmod*{\pmatrixtwo{I_r}{}{}{C_1\tilde C_1^{-2}}}.
\end{align*}
Consider the bottom-right block. If $C_1$ is a scalar matrix, we apply Case 1. Otherwise, we suppose by induction on $n$ that there are
$$O\left(\prod_{i=r+1}^np^{(n-i)(\sigma_i-2\mu_i)}(p,R_i')^{\sigma_i-2\mu_i}\right)$$
solutions for $U_4$. Since $\sigma_1=2\mu_1$, we conclude the proof of the claim in that case.\\
\emph{Case 2.2}: $\sigma_1$ is odd and $C=\diag(p^{\sigma_1}I_r,p^{\sigma_1+1}I_s)$. Let $\sigma=\sigma_1$ and $\mu=\mu_1$. Then we have $\tilde C=\diag(p^\mu I_r,p^{\mu+1}I_s)$ and $C\tilde C^{-1}=p^\mu I_{r+s}$. In particular, $U$ is symmetric. Equation \eqref{eq RU condition} becomes
\begin{align*}
\begin{pmatrix}0&\cdots&0\\&\ddots&\vdots\\&&0\end{pmatrix}&=\tilde C^{-1}RU^t\tilde C+UR^t\pmod*{\pmatrixtwo{pI_r}{}{}{0_s}\Mat_n(\ZZ)}\\
	&=\pmatrixtwo{p^{-\mu}I_r}{}{}{p^{-\mu-1}I_s}\pmatrixtwo{R_1}{R_2}{R_3}{R_4}\pmatrixtwo{U_1}{U_2}{U_2^t}{U_4}\pmatrixtwo{p^\mu I_r}{}{}{p^{\mu+1}I_s}\\
	&\quad+\pmatrixtwo{U_1}{U_2}{U_2^t}{U_4}\pmatrixtwo{R_1^t}{R_3^t}{R_2^t}{R_4^t}\\
	&=\pmatrixtwo{R_1U_1+U_1R_1^t+R_2U_2^t+U_2R_2^t}{U_1R_3^t+U_2R_4^t}\ast\ast.
\end{align*}
Suppose that $(p,R)=1$. Then we do one of the following:
\begin{enumerate}
\item If $(p,R_1)=1$, we fix $U_2$ and apply Case 1 to get $U_1$.
\item If $(p,R_2)=1$, we fix $U_1$ and apply Lemma \ref{lem simpler matrix equations} (5) to get $U_2$.
\item If $(p,R_3)=1$, we fix $U_2$ and apply Lemma \ref{lem simpler matrix equations} (2) to get $U_1$.
\item If $(p,R_4)=1$, we fix $U_1$ and apply Lemma \ref{lem simpler matrix equations} (1) to get $U_2$.
\end{enumerate}
In all cases, there are no condition on $U_3$ and we won $p^r$ over the trivial bound. Therefore we get
$$O\left(p^{r(r+2s-1)/2}(p,R)^r\right)=O\left(\prod_{i=1}^np^{(n-i)(\sigma_i-2\mu_i)}(p,R_i')^{\sigma_i-2\mu_i}\right)$$
solutions for $U$. Note that $R=BW^t$ and $p\ndivides\det(W)$. So $(p,R)=(p,B)$.\\
\emph{Case 2.3}: $\sigma_1$ is odd and $\sigma_n\geq\sigma_1+2$. Let $\sigma=\sigma_1$ and $\mu=\mu_1$. We write $C=\diag(p^\sigma I_r,p^{\sigma+1}I_s,C_1)$. with all the prime powers in $C_1$ larger than $p^{\sigma+1}$ and the convention that $s=0$ if there is no prime power in $C$ equal to $p^{\sigma+1}$. By hypothesis, $C_1$ is non-empty. Then $\tilde C=\diag(\tilde C_0,\tilde C_1)=\diag(p^\mu I_r,p^{\mu+1}I_s,\tilde C_1)$ and $C\tilde C^{-1}=\diag(p^{\sigma-\mu}I_{r+s},C_1\tilde C_1^{-1})$. Note that the prime powers in $\tilde C_0$ and $\tilde C_1$ can be the same. We write
$$R=\pmatrixtwo{R_1}{R_2}{R_3}{R_4},\quad U=\pmatrixtwo{U_1}{U_2}{p^{\mu-\sigma}C_1\tilde C_1^{-1}U_2^t}{U_4}$$
with $R_1$ and $U_1$ blocks of size $r+s$ by $r+s$. Note that $p^{\mu-\sigma}C_1\tilde C_1^{-1}U_2^t=0\pmod{p\Mat_{n-r-s,r+s}(\ZZ)}$ and $U_1$ is symmetric. Equation \eqref{eq RU condition} becomes
\begin{align}\label{eq RU in Case 2.2}
&\begin{pmatrix}0&\cdots&0\\&\ddots&\vdots\\&&0\end{pmatrix}=\tilde C^{-1}RU^t\tilde C+UR^t\pmod{C\tilde C^{-2}\Mat_n(\ZZ)}\nonumber\\
	&=\pmatrixtwo{\tilde C_0^{-1}}{}{}{\tilde C_1^{-1}}\pmatrixtwo{R_1}{R_2}{R_3}{R_4}\pmatrixtwo{U_1}{}{U_2^t}{U_4^t}\pmatrixtwo{\tilde C_0}{}{}{\tilde C_1}+\pmatrixtwo{U_1}{U_2}{}{U_4}\pmatrixtwo{R_1^t}{R_3^t}{R_2^t}{R_4^t}\nonumber\\
	&=\pmatrixtwo{\tilde C_0^{-1}(R_1U_1+R_2U_2^t)\tilde C_0+U_1R_1^t+U_2R_2^t}{\tilde C_0^{-1}R_2U_4^t\tilde C_1+U_1R_3^t+U_2R_4^t}\ast{\tilde C_1^{-1}R_4U_4^t\tilde C_1+U_4R_4^t}\\
	&\qquad\qquad\qquad\qquad\qquad\qquad\qquad\qquad\pmod*{\pmatrixtwo{p^{\sigma-\mu}\tilde C_0^{-1}}{}{}{C_1\tilde C_1^{-2}}}.\nonumber
\end{align}
Consider the bottom-right block. If $C_1\tilde C_1^{-1}$ is a scalar matrix, we apply Case 1. Otherwise, we suppose, by induction on $n$, that there are
$$O\left(\prod_{i=r+s+1}^np^{(n-i)(\sigma_i-2\mu_i)}(p,R_i')^{\sigma_i-2\mu_i}\right)$$
solutions for $U_4$.

Consider the top blocks. We get the equations
\begin{align*}
\begin{pmatrix}0&\cdots&0\\&\ddots&\vdots\\&&0\end{pmatrix}&=\tilde C_0^{-1}(R_1U_1+R_2U_2^t)\tilde C_0+U_1R_1^t+U_2R_2^t &&\pmod*{\pmatrixtwo{pI_r}{}{}{0_s}\Mat_{r+s}(\ZZ)},\\
0&=\tilde C_0^{-1}R_2U_4^t\tilde C_1+U_1R_3^t+U_2R_4^t &&\pmod*{\pmatrixtwo{pI_r}{}{}{0_s}\Mat_{r+s,n-r-s}(\ZZ)}.
\end{align*}

Suppose that $(p,R)=1$. Then we do one of the following:
\begin{enumerate}
\item If $(p,R_1)=1$, consider the first equation. We fix $U_2$ and apply Case 1 or Case 2.1 to get $U_1$ depending if $s=0$ or not. The proof is valid even if the left-hand side of the equation is non-zero.
\item If $(p,R_2)=1$, consider the first equation. We fix $U_1$. We have
\begin{align*}
-\tilde C_0^{-1}&R_1U_1\tilde C_0-U_1R_1^t=\tilde C_0^{-1}R_2U_2^t\tilde C_0+U_2R_2^t\\
	&=\pmatrixtwo{p^{-\mu}I_r}{}{}{p^{-\mu-1}I_s}\begin{pmatrix}R_{21}\\R_{22}\end{pmatrix}\begin{pmatrix}U_{21}^t&U_{22}^t\end{pmatrix}\pmatrixtwo{p^\mu I_r}{}{}{p^{\mu+1}I_s}+\begin{pmatrix}U_{21}\\U_{22}\end{pmatrix}\begin{pmatrix}R_{21}^t&R_{22}^t\end{pmatrix}\\
	&=\pmatrixtwo{R_{21}U_{21}^t+U_{21}R_{21}^t}{U_{21}R_{22}^t}\ast\ast\pmod{\pmatrixtwo{pI_r}{}{}{I_s}\Mat_{r+s}(\ZZ)}
\end{align*}
with $R_{21}$ and $U_{21}$ blocks of size $r$ by $n-r-s$. Recall that we do not consider equations below the diagonal. Fix $U_{22}$. If $(p,R_{22})=1$, we apply Lemma \ref{lem simpler matrix equations} (1) to the top-right block to get $U_{21}$. Otherwise $(p,R_{21})=1$. We apply Lemma \ref{lem simpler matrix equations} (5) to the top-left block to get $U_{21}$. Note that this equation is symmetric, so we can drop the restriction of the equation to the upper-diagonal.
\item If $(p,R_3)=1$, consider the second equation. We fix $U_2$. We have
\begin{align*}
-\tilde C_0^{-1}R_2U_4^t\tilde C_1-U_2R_4^t&=U_1R_3^t=\pmatrixtwo{U_{11}}{U_{12}}{U_{12}^t}{U_{13}}\begin{pmatrix}R_{31}^t\\R_{32}^t\end{pmatrix}\\
	&=\begin{pmatrix}U_{11}R_{31}^t+U_{12}R_{32}^t\\\ast\end{pmatrix}\pmod*{\pmatrixtwo{pI_r}{}{}{I_s}\Mat_{r+s,n-r-s}(\ZZ)}
\end{align*}
with $U_{11}$ a block of size $r$ by $r$ and $R_{31}$ a block of size $n-r-s$ by $r$.  If $(p,R_{31})=1$, then we fix $U_{12}$ and apply Lemma \ref{lem simpler matrix equations} (2) to the top block to get $U_{11}$. Otherwise, $(p,R_{32})=1$. Then we fix $U_{11}$ and apply Lemma \ref{lem simpler matrix equations} (1) to the top block to get $U_{12}$. There is no condition on $U_{13}$.
\item If $(p,R_4)=1$, consider the second equation. We fix $U_1$. We have
$$-\tilde C_0^{-1}R_2U_4^t\tilde C_1-U_1R_3^t=U_2R_4^t=\begin{pmatrix}U_{21}R_4^t\\U_{22}R_4^t\end{pmatrix}\pmod*{\pmatrixtwo{pI_r}{}{}{0_s}\Mat_{r+s,n-r-s}(\ZZ)}$$
with $U_{21}$ a block of size $r$ by $n-r-s$. We apply Lemma \ref{lem simpler matrix equations} (1) to fix $U_{21}$. There is no condition on $U_{22}$.
\end{enumerate}
In any case, we won $p^r$ over the trivial bound for the pair $(U_1,U_2)$. In total, we get
$$O\left(p^{r(2n-r-1)/2}(p,R)^r\cdot\prod_{i=r+s+1}^np^{(n-i)(\sigma_i-2\mu_i)}(p,R_i')^{\sigma_i-2\mu_i}\right)=O\left(\prod_{i=1}^np^{(n-i)(\sigma_i-2\mu_i)}(p,R_i')^{\sigma_i-2\mu_i}\right)$$
solutions for $U$. As before, $R=BW^t$ and $p\ndivides\det(W)$. So $(p,R_i')^{\sigma_i-2\mu_i}=(p,B_i')^{\sigma_i-2\mu_i}$ for $1\leq i\leq r+s$. Recall that $(C,W)$ is a coprime symmetric pair and note that
$$\pmatrixtwo{R_1}{R_2}{R_3}{R_4}=R=BW^t=\pmatrixtwo{B_1}{B_2}{B_2^t}{B_3}\pmatrixtwo{W_1}{}{W_2^t}{W_3^t}=\pmatrixtwo\ast\ast\ast{B_3W_3^t}\pmod{p\Mat_n(\ZZ)}.$$
Since $R_3=B_3W_3^t$ and $p\ndivides\det(W_3)$, we have by induction on $n$ that $(p,R_i')=(p,B_i')$. This concludes the proof of the claim.

Recall Equation \eqref{eq symmetric Gauss sum squared}. Taking the bound of the claim for the number of $U$ and a trivial bound for the number of $V$, we get
$$\abs{H(A,B,W;C\tilde C^{-2})}^2\ll\prod_{i=1}^np^{2(n-i+1/2)(\sigma_i-2\mu_i)}(p,B_i')^{\sigma_i-2\mu_i}.$$
This concludes the proof of the proposition.
\end{proof}

Now, we estimate the number of solutions to the equation $T_3=\bar U_1T_3\bar U_1^t$ appearing the sum over $U_1$ in Proposition \ref{pro section block decomposition}. We need one additional lemma before that.

\begin{lemma}\label{lem diagonalization modulo p^k}
Let $p$ be an odd prime and $k\geq1$ an integer. Let $Q$ be an integral symmetric matrix of size $n$. Then there are $p\ndivides x$, $M\in\Mat_n(\ZZ)$ with $p\ndivides\det(M)$ and $E\in\XXc_{n-r}(\ZZ)$ such that
$$MQM^t=\begin{pmatrix}I_{r-1}\\&x\\&&pE\end{pmatrix}\pmod{p^k}$$
with $r$ being the rank of $Q\pmod p$.
\end{lemma}

\begin{remark}
We can inductively diagonalize $E$. In the end, $Q$ is congruent to a diagonal matrix with prime powers multiplied by invertible elements
\end{remark}

\begin{proof}
If $k=1$, this is true with $E=0$. See Theorem VI.10 in \cite{New74}. More precisely, there is $M$ with $p\ndivides\det(M)$ such that
$$MQM^t=\begin{pmatrix}I_{r-1}\\&x\\&&0_{n-r}\end{pmatrix}\pmod{p\Mat_n(\ZZ)}.$$
For larger prime powers, we use induction. Let $D=\diag(1,\dots,1,x)$ be a matrix of size $r$. Let $k\geq1$ and suppose that we have $M_0$ such that $M_0QM_0^t=\pmatrixtwo D{}{}{pE}\pmod{p^k\Mat_n(\ZZ)}$. Let
$$P=\pmatrixtwo{P_1}{P_2}{P_2^t}{P_3}=p^{-k}\left(\pmatrixtwo D{}{}{pE}-M_0QM_0^t\right)$$
with $P_1$ a $r$ by $r$ block. Write $M=(I_n+p^kN)M_0$ and $N=\smmatrix{N_1}0{N_3}0$. Note that $\det(M)=\det(M_0)\pmod p$. We consider the equation
$$\pmatrixtwo{P_1}{P_2}{P_2^t}0=N\pmatrixtwo D{}{}{pE}+\pmatrixtwo D{}{}{pE}N^t=\pmatrixtwo{N_1D+DN_1^t}{DN_3^t}{N_3D}{}\pmod{p^k\Mat_n(\ZZ)}.$$
A solution for $N$ is the following. We set $N_3=P_2^t\bar D$ with $\bar D$ such that $D\bar D=I_n\pmod{p^k\Mat_n(\ZZ)}$. Let $N_1=(n_{ij})$ and $P_1=(p_{ij})$. We set $n_{ii}=p_{ii}/2$ for $i\leq r-1$, $n_{rr}=p_{nn}/(2x)$ and for $1\leq j<i\leq r$ set $n_{ij}=p_{ij}$. If $1\leq i<j\leq r$, set $n_{ij}=0$. In conclusion, we have
\begin{align*}
MQM^t&=M_0QM_0^t+p^k(NM_0QM_0^t+M_0QM_0^tN^t)&\pmod{p^{2k}\Mat_n(\ZZ)}\\
	&=\pmatrixtwo D{}{}{pE}-p^k\pmatrixtwo{P_1}{P_2}{P_2^t}{P_3}+p^k\pmatrixtwo{P_1}{P_2}{P_2^t}0\\
	&=\pmatrixtwo D{}{}{pE-p^kP_3}.
\end{align*}
By induction on $k$, we have a solution modulo $p^{2k}$ for all $k\geq1$.
\end{proof}

\begin{proposition}\label{pro solutions T=UQU}
Let $p$ be an odd prime. Let $T$ and $Q$ be half-integral symmetric matrices and $C=\diag(p^{\sigma_1},\dots,p^{\sigma_n})$ with $2\leq\sigma_1\leq\dots\leq\sigma_n$. Let $N$ be the number of solutions $U$ to the equation
\begin{align}\label{eq T=UQU}
2T=2UQU^t\pmod{[\tilde C]}
\end{align}
with
$$U\in\tilde X_1(C)=\{U\pmod{\tilde C\Mat_n(\ZZ)}\mid(C,U)\text{ coprime symmetric pair}\}.$$
\begin{enumerate}
\item If $C=p^kI_n$ is scalar, $k\geq1$, and $m=\lfloor\frac k2\rfloor$, then
$$N\ll p^{mn(n-1)/2}(p^m,2Q,2T)^n.$$
\item For all $C$, we have
\begin{align}\label{eq bound solutions T=UQU}
N\ll\prod_{i=1}^np^{(n-i)\mu_i}(p^{\mu_i},2Q_i').
\end{align}
Here $Q_i'$ is the bottom-right block of $Q$ of size $s$ by $s$, where $s$ is the smallest integer with $\sigma_s=\sigma_i$.
\end{enumerate}
\end{proposition}

\begin{remark}
Since $U$ is invertible, it is equivalent to consider the equation $Q=\bar UT\bar U^t$. In other words, we can replace $Q$ by $T$ in Equation \eqref{eq bound solutions T=UQU}. Moreover if $(p^\mu,2Q)\neq(p^\mu,2T)$, then there are no solution. In that case, $K(Q,T;C)=0$ by Proposition \ref{pro section block decomposition}.
\end{remark}

\begin{proof}
We split the proof in two cases, depending on the shape of $C$.\\
\emph{Case 1}: $C=p^kI_n$. Let $m=\lfloor\frac k2\rfloor$. In that case, $U$ is symmetric and $\tilde X_1(C)$ consists of invertible symmetric matrices$\pmod{p^m\XXc_n(\ZZ)}$. We consider first the case where $\rk_p(Q)\geq1$.
\\
\emph{Case 1.1}: $(p,Q)=1$. By Lemma \ref{lem diagonalization modulo p^k}, there is $M\in\Mat_n(\ZZ)$ with $p\ndivides\det(M)$ such that $Q=M\smmatrix D{}{}{pE}M^t$ with $D=\diag(1,\dots,1,x)\in\Mat_r(\ZZ)$, where $r\neq0$ is the rank of $Q\pmod{p\Mat_n(\ZZ)}$. Let $V=M^tUM$ and $P=M^tTM$. Then
$$T=UQU^t\pmod{p^m\Mat_n(\ZZ)}\Leftrightarrow P=V\pmatrixtwo D{}{}{pE}V^t\pmod{p^m\Mat_n(\ZZ)}.$$
Write $P=\smmatrix{P_1}{P_2}{P_2^t}{P_3}$ and $V=\smmatrix{V_1}{V_2}{V_2^t}{V_3}$ with $P_1,V_1$ blocks of size $r$ by $r$. Then the above equation is
$$\pmatrixtwo{P_1}{P_2}{P_2^t}{P_3}=\pmatrixtwo{V_1DV_1+pV_2EV_2^t}{V_1DV_2+pV_2EV_3}{V_2^tDV_1+pV_3EV_2^t}{V_2^tDV_2+pV_3EV_3}.$$
Since $E$ could be 0, we have to fix $V_3$ arbitrarily among the $O(p^{m(n-r)(n-r+1)/2})$ possibilities. Note that $(p,P)=1$ since $V$ is invertible. Suppose that $(p,P_2)=1$. Then clearly $(p,V_1)=(p,V_2)=1$. We conclude that $(p,P_1)=(p,P_3)=1$ and these cases are treated below.

Suppose that $(p,P_1)=1$. By Lemma \ref{lem simpler matrix equations} (4), we have $O(p^{mr(r-1)/2})$ solutions for $V_1$. Moreover $(p,V_1)=1$. Then we can fix $V_2$ in the top-right equation using Hensel's method. More precisely, we have
$$P_2=V_1DV_2+pV_2EV_3\pmod{p^m}.$$
Let $W_0$ be a solution for $V_2$ modulo $p$. Since $(p,V_1)=1$, we have $O(p^{(r-1)(n-r)})$ solutions for $W_0$ by Lemma \ref{lem simpler matrix equations} (1). Suppose that we have a solution $W_1$ modulo $p^l$ for $l\geq1$. Let $W=W_1+p^kW_2$ be a solution modulo $p^{l+1}$. Then
$$p^{-k}(P_2-V_1DW_1-pW_1EV_3)=V_1DW_2\pmod p.$$
There are $O(p^{(r-1)(n-r)})$ solutions for $W_2$ by Lemma \ref{lem simpler matrix equations} (1). By induction, we get $O(p^{m(r-1)(n-r)})$ possibilities for $V_2$.

Suppose that $(p,P_3)=1$. By Lemma \ref{lem simpler matrix equations} (3), we have $O(p^{m(r-1)(n-r)})$ possibilities for $V_2$ in the bottom-right equation. Moreover, $(p,V_2)=1$. Then by Lemma \ref{lem simpler matrix equations} (2), we have $O(p^{mr(r-1)/2})$ possibilities for $V_1$ in the top-right equation.

In total, we have
$$O\left(p^{mr(r-1)/2}\cdot p^{m(r-1)(n-r)}\cdot p^{m(n-r)(n-r+1)/2}\right)=O\left(p^{mn(n-1)/2}\right)$$
solutions for $V$ in that case.\\
\emph{Case 1.2}: $(p,Q)\neq1$. Since $U$ is invertible, we also have $(p,T)\neq1$. More precisely, there is an integer $l$ such that $Q=p^lQ'$ and $T=p^lT'$ and $(p,Q')=(p,T')=1$. If $l\geq m$, then we can not say anything about $U$ and take it arbitrarily. Otherwise, we get the equation
$$T'=UQ'U\pmod{p^{m-l}\Mat_n(\ZZ)}.$$
Applying Case 1.1, we get
$$O\left(p^{(m-l)n(n-1)/2}\right)$$
solutions for $U\pmod{p^{m-l}\Mat_n(\ZZ)}$. We lift these solutions arbitrarily to a solution of the form $U+p^{m-l}V$. There are $O(p^{ln(n+1)/2})$ possibilities for $V$. In total, we get
$$O\left(p^{(m-l)n(n-1)/2}\cdot p^{ln(n+1)/2}\right)=O(p^{mn(n-1)/2}(p^m,Q,T)^n)$$
solutions for $U$.

\ \\
\emph{Case 2}: $C$ not scalar. Let $\sigma=\sigma_1$, $\mu=\mu_1$ and $C=\diag(p^\sigma I_s,C_1)$ with $C_1$ a block of size $n-s$ by $n-s$ with all its prime powers strictly larger that $p^\sigma$. We write
$$Q=\pmatrixtwo{Q_1}{Q_2}{Q_2^t}{Q_3},\quad T=\pmatrixtwo{T_1}{T_2}{T_2^t}{T_3},\quad U=\pmatrixtwo{U_1}{U_2}{p^{-\sigma}C_1U_2^t}{U_3}\in\tilde X_1(C)$$
with $Q_1,T_1,U_1$ blocks of size $s$ by $s$. Note that $U_1$ is symmetric and $p\ndivides\det(U_1)$. Let
$$Y=\pmatrixtwo{I_s}{p^{-\sigma}Y_2C_1}{}{I_{n-s}}$$
with $Y_2=-\bar U_1U_2\pmod{p^\sigma\Mat_{s,n-s}(\ZZ)}$. Then $CYC^{-1}=\smmatrix{I_s}{Y_2}{}{I_{n-s}}$ and $Y^{-1}=\smmatrix{I_s}{-p^{-\sigma}Y_2C_1}{}{I_{n-s}}$. We compute
\begin{align*}
Y^tUCYC^{-1}&=\pmatrixtwo{U_1}{U_1Y_2+U_2}{p^{-\sigma}C_1(Y_2^tU_1+U_2^t)}{U_3+p^{-\sigma}C_1(Y_2^tU_2+U_2^tY_2)+p^{-\sigma}C_1Y_2^tU_1Y_2}\\
	&=:\pmatrixtwo{V_1}{}{}{V_3}\pmod{\tilde C\Mat_n(\ZZ)}.
\end{align*}
Note that the congruence on the last line holds since $Y^t\tilde C=\tilde C(\tilde C^{-1}Y^t\tilde C)$. The last parenthesis is an integral matrix. Note also that $(C,V)$ is a coprime symmetric pair:
$$VC=Y^tUCY=Y^tCU^tY=CV^t.$$
This is equivalent to $(p^\sigma I_s,V_1)$ and $(C_1,V_3)$ being coprime symmetric pairs.

We define a map
\begin{align*}
\tilde X_1(C)&\rightarrow\tilde X_1(p^\sigma I_s)\times\Mat_{s,n-s}(\ZZ/p^\mu\ZZ)\times\tilde X_1(C_1),\\
U&\mapsto (V_1,Y_2,V_3).
\end{align*}
Clearly the map is injective. Since $U_1=V_1$ and $Y$ is invertible, the map is bijective. Let $R=Y^tTY$ and $S=CY^{-1}C^{-1}QC^{-1}Y^{-t}C$. Then
$$T=UQU^t\pmod{[\tilde C]}\Leftrightarrow R=VSV^t\pmod{[\tilde C]}.$$
This gives a bijection between the solutions $U$ of the left-hand side and the solutions $(V_1,Y_2,V_3)$ of the right-hand side. Written in blocks, we get
\begin{align}\label{eq R=VSV}
\pmatrixtwo{R_1}{R_2}{R_2^t}{R_3}&=\pmatrixtwo{V_1S_1V_1^t}{V_1S_2V_3^t}{V_3S_2V_1^t}{V_3S_3V_3^t}\pmod{[\tilde C]},
\intertext{where the blocks are given by}
R=\pmatrixtwo{R_1}{R_2}{R_2^t}{R_3}&=\pmatrixtwo{T_1}{p^{-\sigma}T_1Y_2C_1+T_2}{p^{-\sigma}C_1Y_2^tT_1+T_2^t}{T_3+p^{-\sigma}(C_1Y_2^tT_2+T_2^tY_2C_1)+p^{-2\sigma}C_1Y_2^tT_1Y_2C_1},\nonumber\\
S=\pmatrixtwo{S_1}{S_2}{S_2^t}{S_3}&=\pmatrixtwo{Q_1-Y_2Q_2^t-Q_2Y_2^t+Y_2Q_3Y_2^t}{Q_2-Y_2Q_3}{Q_2^t-Q_3Y_2^t}{Q_3}.\nonumber
\end{align}
In particular, $S_3=Q_3$.

Consider the bottom-right block of Equation \eqref{eq R=VSV}:
$$R_3=V_3Q_3V_3^t\pmod{[\tilde C_1]}.$$
If $\tilde C_1$ is a scalar matrix, we apply Case 1. Otherwise we suppose, by induction on $n$, that the equation has
$$O\left(\prod_{i=s+1}^np^{(n-i)\mu_i}(p^{\mu_i},Q_i')\right)$$
solutions for $V_3$. Here $Q_i'$ is the bottom-right block of $Q$ of size $s$ by $s$, where $s$ is the smallest integer with $\sigma_s=\sigma_i$.

Fix one such solution $V_3$. Suppose that $(p^\mu,Q_1,Q_2,Q_3)=(p^\mu,Q_2,Q_3)$. We consider the top-right block of Equation \eqref{eq R=VSV}. Let $p^k=(p^\mu,S_2)$. Then we have the two equations
\begin{align*}
Q_2&=Y_2Q_3\pmod{p^k},\\
p^{-k}R_2&=V_1(p^{-k}S_2)V_3\pmod{p^{\mu-k}}.
\end{align*}
Note that $(p^\mu,R_2)=p^k$ since $V_1$ and $V_3$ are invertible. By Lemma \ref{lem simpler matrix equations} (2), the second equation has $O(p^{s(s-1)(\mu-k)/2})$ solutions for $V_1$ and there are $O(p^{s(s+1)k/2})$ ways to lift them modulo $p^\mu$.

Let $p^l=(p^k,Q_3)$. Note that $p^l=(p^\mu,Q_2-Y_2Q_3,Q_3)=(p^\mu,Q_2,Q_3)$. Then by Lemma \ref{lem simpler matrix equations} (1), the first equation has $O(p^{s(n-s-1)(k-l)})$ solutions for $Y_2$ and there are $O(p^{s(n-s)(\mu-(k-l))})$ ways to lift them modulo $p^\mu$. In total, we get
\begin{align*}
	&\ll p^{s(s-1)(\mu-k)/2}\cdot p^{s(s+1)k/2}\cdot p^{s(n-s-1)(k-l)}\cdot p^{s(n-s)(\mu-(k-l))}\\
	&\ll p^{s(s-1)\mu/2}\cdot p^{sk}\cdot p^{s(n-s)\mu}\cdot p^{-s(k-l)}\\
	&\ll p^{s(2n-s-1)\mu/2}\cdot p^{sl}\\
	&=p^{s(2n-s-1)\mu/2}(p^\mu,Q_2,Q_3)^s
\end{align*}
solutions for the pair $(V_1,Y_2)$.

Finally, suppose that $(p^\mu,Q_1,Q_2,Q_3)=(p^\mu,Q_1)=p^k$ and $(p^\mu,Q_2,Q_3)>p^k$. Then $(p^\mu,S_1)=(p^\mu,R_1)=p^k$ and the top-left block of Equation \eqref{eq R=VSV} is
$$R_1=V_1S_1V_1^t\pmod{p^\mu}.$$
By Case 1, we have
$$O(p^{s(s-1)\mu/2}(p^\mu,S_1,R_1)^s)=O(p^{s(s-1)\mu/2}(p^\mu,Q_1,Q_2,Q_3)^s)$$
solutions for $V_1$. We fix $Y_2$ arbitrarily among the $O(p^{s(n-s)})$ possibilities. In total, we get
$$O(p^{s(s-1)\mu/2}(p^\mu,Q_1,Q_2,Q_3)^s\cdot p^{s(n-s)})=O(p^{s(2n-s-1)\mu/2}(p^\mu,Q_1,Q_2,Q_3)^s)$$
solutions for the pair $(V_1,Y_2)$.

Note that $Q_i'=Q$ for $i\leq s$. We conclude that
\begin{align*}
N&\ll p^{s(2n-s-1)\mu/2}(p^\mu,Q_1,Q_2,Q_3)^s\cdot\prod_{i=s+1}^np^{(n-i)\mu_i}(p^{\mu_i},Q_i')\\
	&\ll\prod_{i=1}^sp^{(n-i)\mu_i}(p^{\mu_i},Q_i')\prod_{i=s+1}^np^{(n-i)\mu_i}(p^{\mu_i},Q_i')\\
	&=\prod_{i=1}^np^{(n-i)\mu_i}(p^{\mu_i},Q_i').
\end{align*}
\end{proof}

Now, we prove Theorem \ref{thm bound for higher prime powers} using all the estimates above. First note that if there is a $s\leq n$ such that $\sigma_1=\dots=\sigma_s=0$ and $\sigma_{s+1}\neq0$, by Proposition \ref{pro sigma_1=0} we have $K_n(Q,T;C)=K_{n-s}(Q_3,T_3;C_3)$. Moreover
$$\prod_{i=1}^sp^{(n-i+1)\sigma_i}(p^{\mu_i},2Q_i')(p,2Q_i')^{(\sigma_i-2\mu_i)/2}=1$$
So if Theorem \ref{thm bound for higher prime powers} is valid for $\sigma_1\neq0$, then it is valid for $\sigma_1=0$.

Consider Proposition \ref{pro section block decomposition}. Applying Remark (2) after Proposition \ref{pro Gauss sum over matrices}, Proposition \ref{pro symmetric sum over matrices} and Proposition \ref{pro solutions T=UQU}, we get
\begin{align*}
\abs{K(Q,T;C)}&\ll p^{s(n-s/2)}(p,2Q_2,2Q_3)^{s/2}\cdot\prod_{i=s+1}^np^{(n-i+1)\mu_i}\cdot\prod_{i=s+1}^np^{(n-i)\mu_i}(p^{\mu_i},2Q_i')\\
	&\quad\cdot\prod_{i=s+1}^np^{(n-i+1/2)(\sigma_i-2\mu_i)}(p,2Q_i')^{(\sigma_i-2\mu_i)/2}\\
	&=\prod_{i=1}^sp^{n-i+1/2}(p,2Q_2,2Q_3)^{1/2}\prod_{i=s+1}^np^{(n-i+1/2)\sigma_i}(p^{\mu_i},2Q_i')(p,2Q_i')^{(\sigma_i-2\mu_i)/2}.
\end{align*}
For $p\neq2$, this prove the second part of Theorem \ref{thm bound for higher prime powers}. Finally, if $C=p^kI_n$ is scalar with $k\geq2$, the sums over $W$ and $X$ are equal to 1 in Proposition \ref{pro section block decomposition} and $Q=Q_3$, $T=T_3$. We apply the estimates for scalar $C$ in Proposition \ref{pro symmetric sum over matrices} and Proposition \ref{pro solutions T=UQU}. Recall that $(p,2Q)=(p,2Q,2T)$ or Proposition \ref{pro solutions T=UQU} has no solutions. Let $m=\lfloor\frac k2\rfloor$. We get
\begin{align*}
\abs{K(Q,T;C)}&\ll\prod_{i=1}^np^{(n-i+1)m}\cdot p^{mn(n-1)/2}(p^m,2Q,2T)^{n/2}\cdot p^{n^2(k-2m)/2}(p,2Q)^{n(k-2m)/2}\\
	&=p^{kn^2}(p^m,2Q,2T)^{n/2}(p,2Q)^{n(k-2m)/2}.
\end{align*}
For $p\neq2$, this prove the first part of Theorem \ref{thm bound for higher prime powers}.

\subsection{The case \texorpdfstring{$p=2$}{p=2}}

In this section, we adapt the proof of Theorem \ref{thm bound for higher prime powers} to the case of the even prime 2. We did not use that $p$ is odd until Proposition \ref{pro section block decomposition}, except where we used Lemma \ref{lem symmetric character sum over matrices}. We start the section by giving adapted versions of Lemma \ref{lem symmetric character sum over matrices} and Lemma \ref{lem diagonalization modulo p^k}. Then we consider the consequences of these adaptations in Lemma \ref{lem simpler matrix equations}. Finally, we adapt the proof of Theorem \ref{thm bound for higher prime powers} to this case.

\begin{lemma}[Lemma \ref{lem symmetric character sum over matrices} for $p=2$]\label{lem symmetric character sum over matrices for 2}
Let $p=2$. Let $C=\diag(p^{\sigma_1},\dots,p^{\sigma_n})$ with $0\leq\sigma_1\leq\dots\leq\sigma_n$ and let $A\in\XXc_n(\RR)$ be a half-integral matrix. We have
$$\sum_{\substack{D\Mod{C\Mat_n(\ZZ)}\\(C,D)\ \text{sym. pair}}}e(C^{-1}DA)=\delta_{2A=0\Mod{[C]}}\prod_{i=1}^n\delta_{a_{ii}=0\Mod{2^{\mu_i}}}p^{(n-i+1)\sigma_i}.$$
\end{lemma}

\begin{remark}
If we remove the additional equations for the diagonal, we only grow the number of solutions. We will mostly consider only the equation $2A=0\pmod{[C]}$.
\end{remark}

\begin{proof}
In the proof of Lemma \ref{lem symmetric character sum over matrices}, we get the conditions
\begin{align*}
a_{ij}+a_{ji}&=0\pmod{2^{\mu_i}}&(i<j),\\
a_{ii}&=0\pmod{2^{\mu_i}}.
\end{align*}
Since $A$ is half-integral, the equation $2A=0\pmod{[C]}$ recovers the first equation, but only gives
$$2a_{ii}=0\pmod{2^{\mu_i}}.$$
We artificially add the second equation to the result to conclude.
\end{proof}

\begin{lemma}[Lemma \ref{lem diagonalization modulo p^k} for $p=2$]\label{lem diagonalization modulo 2^k}
Let $p=2$ and $k\geq1$ an integer. Let $H=\smmatrix0110$. Let $Q\in\XXc_n(\RR)$ be a half-integral matrix. Then there are $M\in\Mat_n(\ZZ)$ and $E\in\XXc_{n-r}(\ZZ)$ and either $D=\diag(d_1,\dots,d_r)$ with $D=I_r\pmod 2$ or $D'=\diag(H_1,\dots,H_{r/2})$ with $D'=\diag(H,\dots,H)\pmod 2$ such that
$$2MQM^t=\pmatrixtwo D{}{}{2E}\pmod{2^k}\quad\text{or}\quad2MQM^t=\pmatrixtwo{D'}{}{}{2E}\pmod{2^k},$$
with $r$ being the rank of $2Q\pmod 2$. Also, $r$ is even in the second case.
\end{lemma}

\begin{proof}
Consider $\tilde Q=2Q$. It is a symmetric integral matrix with even coefficients on the diagonal. Theorem IV.11 in \cite{New74} says that there exists $M$ with $2\ndivides\det(M)$ such that
$$M\tilde QM^t=\pmatrixtwo{I_r}{}{}{0_{n-r}}\ \text{or}\ M\tilde QM^t=\begin{pmatrix}H\\&\ddots\\&&H\\&&&0_{n-r}\end{pmatrix}\pmod{2\Mat_n(\ZZ)},$$
where the last matrix contains $r/2$ copies of $H$ ($r$ is even). In the first case, the proof goes essentially the same way except that we have different diagonal elements. Here we set $n_{ii}=0$. Then the final matrix $D=\diag(d_1,\dots,d_r)$ is such that $d_i=1\pmod 2$.

In the second case, let $D$ be the diagonal matrix given by $r/2$ copies of $H$. We want to find $N_1,N_3$ such that
\begin{align*}
P_1&=N_1D+DN_1^t\pmod{2^k\Mat_r(\ZZ)},\\
P_2&=DN_3^t\pmod{2^k\Mat_{r,n-r}(\ZZ)}.
\end{align*}
Since $2\ndivides\det(D)$, we have $N_3=P_2^t\bar D$ with $D\bar D=I_n\pmod{2^k}$. Writing $P_1=(P_{ij})$ and $N_1=(N_{ij})$ where each $P_{ij},N_{ij}$ is a 2 by 2 blocks. We get
$$P_{ij}=(N_1D+DN_1^t)_{ij}=N_{ij}H+HN_{ji}.$$
Let $i<j$. We set $N_{ij}=P_{ij}H$ and $N_{ji}=0$. For $i=j$, write $P_{ii}=\smmatrix{p_1}{p_2}{p_2}{p_4}$ and $N_{ii}=\smmatrix{n_1}{n_2}{n_3}{N_3}$. We get
$$\pmatrixtwo{p_1}{p_2}{p_2}{p_4}=\pmatrixtwo{n_2}{n_1}{n_4}{n_3}+\pmatrixtwo{n_2}{n_4}{n_1}{n_3}=\pmatrixtwo{2n_2}{n_1+n_4}{n_1+n_4}{2n_3}\pmod{2^k\Mat_2(\ZZ)}.$$
We set $n_2=n_3=n_4=0$ and $n_1=p_2$. As in the first case, we get a matrix $D=\diag(H_1,\dots,H_{r/2})$ with $D=\diag(H,\dots,H)\pmod 2$.
\end{proof}

Now, we adapt Lemma \ref{lem simpler matrix equations}. The proof for (1)--(4) are still valid. For (3) and (4), we have to show the same statement with $D=\diag(H_1,\dots,H_{r/2})$ as in Lemma \ref{lem diagonalization modulo 2^k}. For (5), we have to adapt the proof. From now, we write $v(a)$ for the 2-adic valuation of $a\in\ZZ$. First, we need the following result.

\begin{lemma}[\cite{DMSMR17}]\label{lem quadratic equation mod 2}
Let $a,b,c$ be integers with $v((a,2b,c))<k$. The quadratic equation
$$ax^2+2bx+c=0\pmod{2^k}$$
has at most $2^{v(b^2-ac)/2+2}$ solutions. Moreover, if $v(a)\neq v((a,2b,c))$, then the equation has at most $2^{v((a,b,c))}$ solutions.
\end{lemma}

\begin{remark}
Note that the first case is always worse than the second since $v(b^2-ac)\geq2v((a,b,c))$.
\end{remark}

\begin{proof}
Let $t=v(a,b,c)$. The article states that the number of solutions is at most $2^{v((a,b,c))+D/2+2}$ solutions, where
$$D=v((b/2^t)^2-ac/2^{2t})=v(b^2-ac)-2v((a,b,c))$$
is the discriminant of the reduced equation. By inserting the second equation in the first, we conclude. Moreover, if $v(a)\neq v((a,2b,c))$, then we are only in the cases of Table 1 where we have $2^{v((a,b,c))}$ solutions.
\end{proof}

\begin{proof}[Proof of Lemma \ref{lem simpler matrix equations} for $p=2$]
First, we consider the equation $R=UH_iU^t$ for 2 by 2 matrices, with $R$ and $H_i$ symmetric and $H_i=H\pmod2$. In coordinates, we get
\begin{align}\label{eq R=UHU}
\pmatrixtwo{r_1}{r_2}{r_2}{r_4}&=\pmatrixtwo{u_1}{u_2}{u_3}{u_4}\pmatrixtwo{h_1}{h_2}{h_2}{h_4}\pmatrixtwo{u_1}{u_3}{u_2}{u_4}\nonumber\\
	&=\pmatrixtwo{h_1u_1^2+2h_2u_1u_2+h_4u_2^2}{h_1u_1u_3+h_2(u_1u_4+u_2u_3)+h_4u_2u_4}\ast{h_1u_3^2+2h_2u_3u_4+h_4u_4^2}\pmod{2^k}.
\end{align}

We fix $u_2\pmod{2^k}$ and consider the top-left block in Equation \eqref{eq R=UHU}. If $2u_2=0\pmod{2^k}$, we have $O(2^k)$ possibilities for $u_1$. Suppose that $2u_2\neq0\pmod{2^k}$. Recall that $2\ndivides h_2$. The equation with respect to $u_1$ has discriminant
$$D=(h_2u_2)^2-h_1(h_4u_2^2-r_1)=u_2^2(h_2^2-h_1h_4)+h_1r_1.$$
Let $v=v(D)$. Then we have $O(2^{v/2})$ solutions for $u_1$. We get the additional equation
$$u_2^2(h_2^2-h_1h_4)=-h_1r_1\pmod{2^v}$$
Since $2\ndivides\det(H)$, we have $O(2^{v/2})$ solutions for $u_2$ by the claim in Lemma \ref{lem simpler matrix equations}. There are $O(2^{k-v})$ ways to lift the solution modulo $2^k$. Then the number of solutions for the pair $(u_1,u_2)$ is bounded by
\begin{align}\label{eq solutions quadratic equation in two variables}
\ll\sum_{v=0}^k2^{k-v+v/2+v/2}\ll k2^k.
\end{align}

Note that $UHU^t=\det(U)H\pmod2$. In particular, $(2,R)=1$ if and only if $2\ndivides r_2$. In that case, we have the following equation
$$r_2=h_1u_1u_3+h_2(u_1u_4+u_2u_3)+h_4u_2u_4\pmod{2^k}.$$
We see that $2\ndivides u_1,u_4$ or $2\ndivides u_2,u_3$. Suppose without loss of generality that the first holds. If $2\divides u_1,u_4$, exchange their roles in the following. We fix $u_1$ and $u_4$. Then in the top-left block of Equation \eqref{eq R=UHU}, the discriminant for $u_2$ has valuation
$$v(u_1^2(h_2^2-h_1h_4)+h_4r_1)=0.$$
So we have $O(1)$ solutions for $u_2$. The same is true for the bottom-right block of Equation \eqref{eq R=UHU}. We get $O(1)$ solutions for $u_3$. In total, we have $O(p^{2k})$ solutions for $U$. We conclude that if $(2,R)=1$, the equation $R=UHU^t\pmod{2^k}$ has $O(2^{2k})$ solutions for $U$.

Finally, we consider the case where $U$ is symmetric. The equation is the same, except that $u_2=u_3$. Again if $(2,R)=1$, then $2\ndivides r_2$. We get the equation
$$r_2=h_1u_1u_2+h_2(u_1u_4+u_2^2)+h_4u_2u_4\pmod{2^k}.$$
As in the asymmetric case, either $2\ndivides u_1,u_4$ or $2\ndivides u_2$. In the second case, we can fix $u_2$ and do the same as before. If $2\ndivides u_1,u_4$, fix $u_1\pmod{2^k}$. Then we saw in the asymmetric case that $u_2$ has $O(1)$ solutions. We get the equation
$$r_2-h_1u_1u_2-h_2u_2^2=u_4(h_2u_1+h_4u_2)\pmod{2^k}.$$
Since $2\ndivides h_2$ and $2\divides h_4$, this fixes $u_4$. We conclude that if $U$ is symmetric and $(2,R)=1$, the equation $R=UHU^t\pmod{2^k}$ has $O(2^k)$ solutions for $U$.

Now, we prove what is missing for (3), (4) and (5). For (3) and (4), we suppose that $m$ and $n$ are even and we consider $D=\diag(H_1,\dots,H_{m/2})$. We write $R=(R_{ij})$, $U=(U_{ij})$ in 2 by 2 blocks.
\begin{enumerate}
\addtocounter{enumi}{2}
\item\emph{Case $m=2$}: for $1\leq i,j\leq n/2$, we have
$$R_{ij}=U_{i1}H_1U_{j1}^t\pmod{2^k}.$$
If $i_0,j_0$ is such that $(2,R_{i_0j_0})=1$, then $(2,U_{i_01})=1$ and so $(2,R_{i_0i_0})=1$. We saw above that we have $O(2^{2k})$ solutions for $U_{i_01}$ in that case. Then consider the equation
$$R_{i_0j}=U_{i_01}H_1U_{j1}\pmod{2^k}.$$
for $j\neq i_0$. By Lemma \ref{lem simpler matrix equations} (1), we have $O(2^{2k})$ choices for $U_{j1}$ since $2\ndivides\det(H_1)$. In total, we get $O(2^{2k\cdot n/2})$ choices for $U$.\\
\emph{Case $m=4$}: for $1\leq i\leq n/2$, we have
$$R_{ii}=U_iGU_i^t+V_iHV_i^t$$
Let $U_i=(u_j)$, $V_i=(v_j)$ and $G,H$ with coordinates numbered as in Equation \eqref{eq R=UHU}. In the top-left block, we have the equation
$$r_1=g_1u_1^2+2g_2u_1u_2+g_4u_2^2+h_1v_1^2+2h_2v_1v_2+h_4v_2^2\pmod{2^k}.$$
Let $r=r_1-g_1u_1^2-2g_2u_1u_2-g_4u_2^2$ and $v=v(r)$. The number of solutions of 
$$r_1=g_1u_1^2+2g_2u_1u_2+g_4u_2^2\pmod{2^v}$$
is $O((v+1)2^v)$ as seen above. They lift in $O(2^{2k-2v})$ ways, so we have $O((v+1)2^{2k-v})$ solutions for $u_1,u_2$ for a fixed $v$. Now we consider the equation
$$r=h_1v_1^2+2h_2v_1v_2+h_4v_2^2\pmod{2^k}.$$
Let $t=v(v_2)$. We have $O(2^{k-t})$ choices for $v_2$ for a fixed $t$. Suppose that $t<v(h_1)-1$ or $v(h_4v_2^2-r)<v(h_1)$. Then we are in the second case of Lemma \ref{lem quadratic equation mod 2} for the equation with respect to $v_1$. Therefore, we have
$$O(2^{\min\{t,v(h_4v_2^2-r)\}})$$
solutions for $v_1$ in the last equation. Note that if $t\leq v$, the minimum is $t$ and otherwise it is $v$. We see that the number of solutions for the pair $(v_1,v_2)$ is in that case
\begin{align*}
	&\ll\sum_{v=0}^k(v+1)2^{2k-v}\left(\sum_{t=0}^v2^{k-t}\cdot2^t+\sum_{t=v+1}^{v(h_1)-2}2^{k-t}\cdot2^v\right)\\
	&\ll\sum_{v=0}^k(v+1)2^{2k-v}((v+1)2^k+2^k)\\
	&\ll2^{3k}.
\end{align*}

Suppose that $t\geq v(h_1)-1$ and $v(h_4v_2^2-r)\geq v(h_1)$. This implies that 
$$v=v(r-h_4v_2^2+h_4v_2^2)\geq\min\{v(h_1),2t+1\}\geq v(h_1)-1.$$
Let $D=v_2^2(h_2^2-h_1h_4)+h_1r$ and $d=v(D)$. The number of solutions for $v_1$ is $O(2^{d/2})$. Since $d>2(v(h_1)-1)$, we have the additional equation
$$v_2^2(h_2^2-h_1h_3)=-h_1r\pmod{2^d}.$$
This has at most $O(2^{(v(h_1)+v)/2})$ solutions by the claim in the original proof of Lemma \ref{lem simpler matrix equations}. There are $O(2^{k-d})$ ways to lift them modulo $2^k$. In conclusion the number of solutions for $(u_1,u_2,v_1,v_2)$ is
\begin{align*}
\sum_{v=v(h_1)-1}^k&(v+1)2^{2k-v}\sum_{d=2(v(h_1)-1)}^k2^{d/2}\cdot2^{(v(h_1)+v)/2}\cdot2^{k-d}\\
	&\ll2^{3k}\sum_{v=v(h_1)+1}^k(v+1)2^{-v/2}\sum_{d=2(v(h_1)-1)}^k2^{v(h_1)/2-d/2}\\
	&\ll2^{3k}.
\end{align*}
Therefore the number of solutions for the pair $(U_i,V_i)$ is $O(2^{6k})$ and we conclude by summing over $i=1,\dots,n/2$.\\
\emph{Case $m\geq6$}: Fix $U_{ij}\pmod{2^k}$ for $1\leq i\leq n/2$ and $1\leq j\leq(m-4)/2$. Then
$$U_{i,m-1}H_{m-1}U_{i,m-1}^t+U_{im}H_mU_{im}^t=R_{ii}-\sum_{j=1}^{(m-4)/2}U_{ij}H_jU_{ij}^t.$$
By Case $m=4$, there are at most $O(2^{6k})$ solutions for the pair $(U_{i,m-1},U_{im})$. In total, we have at most $O(2^{k(m-1)n})$ solutions.

\item\emph{Case $n=2$}: the equation is $R=UH_1U^t$ with $U$ symmetric. This was solve at the beginning.\\
\emph{Case $n=4$}: Let $i=1,2$. We have
$$R_{ii}=U_{i1}GU_{i1}^t+U_{i2}HU_{i2}^t\pmod{2^k}.$$
Let $v_1=v((u_{13},u_{14}))$, $v_2=v((u_{23},u_{24}))$ and $v=\min\{v_1,v_2\}$ be fixed. Then $U_{12}=0\pmod{2^v}$. From the equation $R_{11}=U_{11}GU_{11}^t+U_{12}HU_{12}^t$, we get in coordinates the equations
\begin{align*}
r_{11}&=g_1u_{11}^2+2g_2u_{11}u_{12}+g_4u_{12}^2 &&\pmod{2^{2v_1}},\\
r_{12}&=g_1u_{11}u_{12}+g_2(u_{11}u_{22}+u_{12}^2)+g_4u_{12}u_{22} &&\pmod{2^{2v}},\\
r_{22}&=g_1u_{12}^2+2g_2u_{12}u_{22}+g_4u_{22}^2 &&\pmod{2^{2v_2}}.
\end{align*}
Our goal is to show that the number of solutions for $(u_{11},u_{12},u_{22})$ is
$$O((v_1+v_2+1)^22^{3k-v-v_1-v_2}).$$
Suppose that $2u_{12}\neq0\pmod{2^{2v}}$. Consider the first and the last equation. Then the discriminant for the other variable than $u_{12}$ in them is respectively
\begin{align*}
D_1&=u_{12}^2(g_2^2-g_1g_4)+g_1r_{11},\\
D_2&=u_{12}^2(g_2^2-g_1g_4)+g_4r_{22}.
\end{align*}
Let $d_i=v(D_i)$, $i=1,2$. Then we get the additional equation $D_i=0\pmod{2^{d_i}}$ for $u_{12}$. Let $d=\max\{d_1,d_2\}$. We have $O(2^{d/2})$ solutions for $u_{12}\pmod{2^d}$ by the claim in the original proof of Lemma \ref{lem simpler matrix equations}. We can lift the solutions modulo $2^k$ in $O(2^{k-d})$ ways. Then the number of solutions for $u_{ii}\pmod{2^{2v_i}}$ is $O(2^{d_i/2})$. There are $O(2^{k-2v_i})$ ways to lift these solutions modulo $2^k$. Therefore the number of solutions for $U_{11}\pmod{2^k}$ in that case is
\begin{align*}
	&\ll\sum_{d_1=0}^{2v_1}2^{k-2v_1+d_1/2}\left(\sum_{d_2=0}^{d_1}2^{k-2v_2+d_2/2}2^{k-d_1/2}+\sum_{d_2=d_1+1}^{2v_2}2^{k-2v_2+d_2/2}2^{k-d_2/2}\right)\\
	&\ll\sum_{d_1=0}^{2v_1}2^{k-2v_1+d_1/2}2^{2k-2v_2}(1+2v_2)\\
	&\ll(v_2+1)2^{3k-v_1-2v_2}.
\end{align*}

Suppose now that $2u_{12}=0\pmod{2^{2w}}$ with $w=\max\{v_1,v_2\}$. Then we have $O(1)$ solutions for $u_{12}\pmod{2^{2w}}$. Consider the equation
$$r_{12}=g_2(u_{11}u_{22}+u_{12}^2)\pmod{2^{2v}}.$$
The product $u_{11}u_{22}\pmod{2^{2v}}$ is fixed by the equation since $2\ndivides g_2$. Let $t$ be the maximum between its valuation and $2v$. Then $v(u_{11})+v(u_{22})\geq t$. The inequality takes into account the case $t=2v$. Let $r=v(u_{11})$. Dividing the above equation by $2^r$, we can invert $2^{-r}u_{11}$ and fix $u_{22}\pmod{2^{2v-r}}$. There are $O(2^r)$ ways to lift $u_{22}$ modulo $2^{2v}$. Then the number of solutions for the pair $(u_{11},u_{22})$ modulo $2^{2v}$ is
$$\ll\sum_{t=0}^{2v}\sum_{r=0}^t2^{2v-r}2^r\ll\sum_{t=0}^{2v}(t+1)2^{2v}\ll(v+1)^22^{2v}.$$
There are $O(2^{3k-2w-4v})$ ways to lift the solutions modulo $p^k$. We get $O((v+1)^22^{3k-2w-2v})$ solutions for $U_{11}$ in that case.

Suppose that $v_1<v_2$ and that $2u_{12}=0\pmod{2^{2v_1}}$ but $2u_{12}\neq0\pmod{2^{2v_2}}$. Let $t=v(u_{22})$. We have the two equations
\begin{align*}
r_{12}&=g_2(u_{11}u_{22}+u_{12}^2)\pmod{2^{2v_1}}\\
r_{22}&=g_1u_{12}^2+2g_2u_{12}u_{22}+g_4u_{22}^2\pmod{2^{2v_2}}.
\end{align*}
For any value of $t$, we $O(2^k)$ solutions for $u_{11}$, $O(2^{k-2v_1})$ solutions for $u_{12}$ and $O(2^{k-t})$ solutions for $u_{22}$. We apply these bounds for $t\geq v_2$. Suppose that $t\leq v_2-1$. Since $2u_{12}=0\pmod{2^{2v_1}}$, the second equation implies that $r_{22}=g_4u_{22}^2\pmod{2^{2v_1}}$. The discriminant of the second equation with respect to $u_{12}$ is
$$D=g_2^2u_{22}^2-g_1(g_4u_{22}^2-r_{22})=g_2^2u_{22}^2\pmod{2^{2v_1}}.$$
Suppose that $t\leq v_1-1$. Then $v(D)=2t$. If $t\geq v_1$, then $v(D)\leq 2v_2$. We have $O(2^{v(D)/2})$ solutions for $u_{12}\pmod{2^{2v_2}}$. There are $O(2^{k-2v_2})$ ways to lift $u_{12}$ modulo $2^k$. We have $O(2^{k-t})$ solutions for $u_{22}$ in any case. Finally, if $t\leq 2v_1$, the first equation implies that
$$g_22^{-t}u_{22}u_{11}=2^{-t}(r_{12}-g_2u_{12}^2)\pmod{2^{2v_1-t}}.$$
Since $2^{-t}u_{22}$ is invertible, $u_{11}$ is fixed. There are $O(2^{k-2v_1+t})$ ways to lift it. Note that this estimate is trivial for $t>2v_1$. In total, we have
\begin{align*}
	&\ll\sum_{t=0}^{v_1-1}2^{k-t}2^{k-2v_2+t}2^{k-2v_1+t}+\sum_{t=v_1}^{v_2-1}2^{k-t}2^{k-v_2}2^{k-2v_1+t}+\sum_{t=v_2}^k2^{k-t}2^{k-2v_1}2^k\\
	&\ll2^{3k-v_1-2v_2}+(v_2+1)2^{3k-2v_1-v_2}+2^{3k-2v_1-v_2}\\
	&\ll(v_2+1)2^{3k-2v_1-v_2}.
\end{align*}
If $v_1>v_2$ and $2u_{12}=0\pmod{2^{2v_2}}$ but $2u_{12}\neq0\pmod{2^{2v_1}}$, we can exchange the role of $v_1$ and $v_2$ in the above proof and get a similar bound.

Once $U_{11}$ is fixed, consider $U_{12}$. We have the equations
\begin{align*}
s_1&=h_1u_{13}^2+2h_2u_{13}u_{14}+h_4u_{14}^2\pmod{2^k},\\
s_2&=h_1u_{23}^2+2h_2u_{23}u_{24}+h_4u_{24}^2\pmod{2^k},
\intertext{with}
s_1&:=r_{11}-(g_1u_{11}^2+2g_2u_{11}u_{12}+g_4u_{12}^2),\\
s_2&:=r_{22}-(g_1u_{12}^2+2g_2u_{12}u_{22}+g_4u_{22}^2).
\end{align*}
Recall that $v_1=v((u_{13},u_{14}))$. Note that $v(s_1)\geq 2v_1+1$. If $v_1\geq k-1$, then	 we have $O(1)$ choices for $u_{13}$ and $u_{14}$. Suppose that $v_1\leq k-2$. Suppose also $v(u_{14})=v_1$. Otherwise inverse the roles of $u_{13}$ and $u_{14}$ in what follows. The discriminant of the first equation with respect to $u_{13}$ is $D=u_{14}^2(h_2^2-h_1h_4)+h_1s_1$. Then $v(D)=2v_1$. We saw before Equation \eqref{eq solutions quadratic equation in two variables} that the number of solutions for the pair $(u_{13},u_{14})$ once the valuation of the determinant is fixed is $O(2^k)$. ref Doing the same with the pair $(u_{23},u_{24})$, we get $O(2^{2k})$ solutions for $U_{12}\pmod{2^k}$.

Finally, we have the equation
$$R_{12}-U_{11}GU_{12}=U_{12}HU_{22}\pmod{2^k\Mat_2(\ZZ)}.$$
We fixed $U_{11}$ and $U_{12}$. Both sides are divisible by $v$ and $(2,2^{-v}U_{12})=1$. By Lemma \ref{lem simpler matrix equations} (2), we get $O(2^{k-v})$ solution for $U_{22}\pmod{2^{k-v}\Mat_2(\ZZ)}$. We have $O(2^{3v})$ ways to lift the solutions modulo $2^k$. In total, we have $O(2^{k+2v})$ solutions for $U_{22}$. Summing over $v_1$ and $v_2$, the number of solutions for $U$ is
\begin{align*}
	&\ll2^{2k}\sum_{v_1=0}^k\left(\sum_{v_2=0}^{v_1}(v_1+1)^22^{3k-v_1-2v_2}2^{k+2v_2}+\sum_{v_2=v_1+1}^k(v_2+1)^22^{3k-2v_1-v_2}2^{k+2v_1}\right)\\
	&\ll2^{2k}\sum_{v_1=0}^k2^{4k}((v_1+1)^32^{-v_1}+(v_1+1)^22^{-v_1})\\
	&\ll2^{6k}.
\end{align*}
\emph{Case $n\geq6$}: let $1\leq i\leq n/2$. We have
$$R_{ii}=\sum_{j=1}^{n/2}U_{ij}H_jU_{ij}^t\pmod{2^k}$$
Fix $U_{ij}\pmod{2^k}$ for $1\leq i\leq j\leq(n-4)/2$. Then
$$U_{i,n-1}H_{n-1}U_{i,n-1}^t+U_{in}H_nU_{in}^t=R_{ii}-\sum_{j=1}^{(n-4)/2}U_{ij}H_jU_{ij}^t\pmod{2^k}.$$
Consider $i$ in increasing order. We saw in (1) that we have $O(2^{4k})$ solutions for the pair $(U_{i,n-1},U_{in})$ once the rest is fixed. Finally for $(n-4)/2\leq i\leq n/2$, we get a 2 by 2 block matrix equation that corresponds to the case $n=4$. In total, we get
$$O(2^{k(n-4)(n-3)/2}\cdot2^{6k(n-4)/2}\cdot2^{6k})=O(2^{kn(n-1)/2})$$
solutions for $U$.

\item The proof is coherent. We only lose a power of 2 in the second display when we evaluate
$$T_{j_0j_0}=(QU^t+UQ^t)_{j_0j_0}\pmod{2^k}.$$
But in the only application in Proposition \ref{pro symmetric sum over matrices}, we have an additional equation (see below)
$$T_{j_0j_0}=(QU^t+UQ^t)_{j_0j_0}\pmod{2^{k+1}}.$$
So we get the same bound from this equation. The rest of the proof does not change.
\end{enumerate}
\end{proof}

Now, we can consider the proofs of Propositions \ref{pro Gauss sum over matrices}, \ref{pro symmetric sum over matrices} and \ref{pro solutions T=UQU} for $p=2$.\\
\emph{Proposition \ref{pro Gauss sum over matrices}}: we consider $A$ to be half-integral and $B_1$ to be symmetric half-integral, which is the case in our application. Then $2A$ and $2B_1$ are integral and $\tr(MB_1)$ is integral for any symmetric matrix $M\in\XXc_n(\ZZ)$. With this in mind, the proof goes the same way. We get the same results (with the condition on $2A$ in the second case).\\
\emph{Proposition \ref{pro symmetric sum over matrices}}: we consider $A$ and $B$ to be half-integral symmetric matrices. The proof goes the same way. We get the equations
\begin{align*}
2(2^{\mu_j-\mu_i}m_{ij}+m_{ji})&=0\pmod{2^{\sigma_i-2\mu_i}},&(i<j),\\
m_{ii}&=0\pmod{2^{\sigma_i-2\mu_i}}.
\end{align*}
If we drop the second equation, we get Equation \eqref{eq RU condition} with $R$ replaced by $2R$. The proof in Case 1 is coherent. The second equation makes the above proof of Lemma \ref{lem simpler matrix equations} (5) valid. The rest of the proof goes the same way. We use that $2R=2BW^t$ and get the same result with $2B$ instead of $B$.\\
\emph{Proposition \ref{pro solutions T=UQU}:} after applying Lemma \ref{lem symmetric character sum over matrices for 2}, we get an additional congruence for the diagonal elements. If we drop it, the proof goes the same way, with two different possibilities for $D$ in Case 1. Since the bound from Lemma \ref{lem simpler matrix equations} (3), (4) are the same for the two different $D$, we obtain the same result in Case 1. The rest of the proof is the same with $2Q,2T$ instead of $Q,T$ and $2Q_i'$ in the bound instead of $Q_i'$. We conclude that the same results hold. 

Now that we showed that all estimates hold for $p=2$, the rest of the proof of Theorem \ref{thm bound for higher prime powers} for $p=2$ goes the same way.

\section{Application}\label{sec applications}
In this section, we prove Theorem \ref{thm application intro}. First, we prove a non-trivial bound for a Kloosterman sum with a general $C$ and give a bound on Fourier coefficients of smooth functions.

\begin{proposition}\label{pro final bound}
Let $C\in\Mat_n(\ZZ)$ with $\det(C)\neq0$. Let $Q,T$ be half-integral symmetric matrices. Let $\epsilon>0$. We have
$$K_n(Q,T;C)\ll_{n,\epsilon}c_n^{\epsilon}c_1^{n-1/2}(c_1,2Q,2T)^{3/2}\prod_{i=2}^nc_i^{n-i+1}$$
where the implicit constant only depends on $n$ and $\epsilon$. Here $c_1\divides\cdots\divides c_n$ are the elementary divisors of $C$.
\end{proposition}

\begin{proof}
Note that the result is true for $n=1$ by the Weil bound. Suppose that $C$ is not in its Smith normal form $C'$. There are $U,V\in\GL_n(\ZZ)$ such that $C'=U^tCV$ and by Lemma \ref{lem factorization Smith normal form} we have
$$K(Q,T;C)=K(Q[U],T[V];C').$$
Note that $(c,2Q[U],2T[V])=(c,2Q,2T)$ for all $c\in\ZZ$ since $U,V$ are invertible. So without loss of generality, we suppose that $C$ is in its Smith normal form.

Let $C=\diag(c_1,\dots,c_n)$ with $c_1\divides\cdots\divides c_n$. Suppose first that $c_i=p^{\sigma_i}$ for a fixed prime $p$ and $0\leq\sigma_1\leq\dots\leq\sigma_n$. If $\sigma_i\leq 1$ for all $i=1,\dots,n$, Theorem \ref{thm bound Toth and Zabradi} combined with Proposition \ref{pro sigma_1=0} gives the bound
\begin{align}\label{eq Toth bound for application}
K(Q,T;C)\ll_nc_1^{n-1/2}(c_1,2Q,2T)^{1/2}\prod_{i=2}^nc_i^{n-i+1}.
\end{align}
If $\sigma_n\geq2$, Theorem \ref{thm bound for higher prime powers} gives the bound
$$K(Q,T;C)\ll_nc_1^{n-1+1/2}(c_1,2Q_1')^{3/2}\prod_{i=2}^nc_i^{n-i+1,}$$
with $Q_1'=Q$ except if $c_1=p$. This is because $(c_1'',2Q_1')\leq(c_1',2Q_1')\leq(c_1,2Q_1')$. The same bound is valid when replacing $Q$ by $T$ thanks to Lemma \ref{lem symmetry in QT}. So we can replace $(c_1,2Q_1')$ by $(c_1,2Q_1',2T_1')$. In the case where $c_1=p$, let $C=\diag(pI_s,C_1)$ with all the prime powers in $C_1$ at least $p^2$. We have $(p,2Q_1',2T_1')=(p,2Q_2,2Q_3,2T_2,2T_3)$ where $Q,T$ are split into blocks of the same size as $C$. Suppose that
$$(p,2Q_2,2Q_3,2T_2,2T_3)=p.$$
Then in Proposition \ref{pro section block decomposition}, the sum over $X$ is trivial. Thus the sum over $W$ is $K_s(Q_1,T_1;pI_s)$. Applying the bound of Equation \eqref{eq Toth bound for application} (or the Weil bound if $n=1$), we get
$$K(Q,T;C)\ll c_1^{n-1/2}(c_1,2Q_1,2T_1)^{1/2}\prod_{i=2}^nc_i^{n-i+1}.$$
In conclusion, the following bound holds for any $C=\diag(p^{\sigma_1},\dots,p^{\sigma_n})$:
\begin{align}\label{eq general bound for p}
K(Q,T;C)\leq E_nc_1^{n-1/2}(c_1,2Q,2T)^{3/2}\prod_{i=2}^nc_i^{n-i+1}
\end{align}
with $E_n$ a fixed constant that only depends on $n$.

Let $\omega(c)$ be the number of prime divisors of $c$ (without multiplicity). We proved Equation \eqref{eq general bound for p} in the case $\omega(c_n)=1$. Now, we prove it for $\omega(c_n)>1$ working by induction. Let $p\divides c_n$ and write $C=FG=\diag(f_1,\dots,f_n)\diag(g_1,\dots,g_n)$ with $f_i=(p^\infty,c_i)$. By Lemma \ref{lem factorization prime powers}, we have
$$K(Q,T;C)=K(Q_F,T;F)\cdot K(Q_G,T;G)$$
with $(p,Q_F)=(p,Q)$ and $(q,Q_G)=(q,Q)$ for all prime $q\divides g_n$. Applying Equation \eqref{eq general bound for p} and induction on $\omega(g_n)$, we get
\begin{align*}
K(Q,T;C)&\leq E_nf_1^{n-1/2}(f_1,2Q,2T)^{3/2}\prod_{i=2}^nf_i^{n-i+1}\\
	&\quad\cdot E_n^{\omega(c_n)-1}g_1^{n-1/2}(g_1,2Q,2T)^{3/2}\prod_{i=2}^ng_i^{n-i+1}\\
	&=E_n^{\omega(c_n)}c_1^{n-1/2}(c_1,2Q,2T)^{3/2}\prod_{i=2}^nc_i^{n-i+1}.
\end{align*}
Finally $E_n^{\omega(c_n)}\ll_{n,\epsilon}c_n^\epsilon$ for all $\epsilon>0$ by the divisor bound. This concludes the proof of the proposition.
\end{proof}

\begin{lemma}[\cite{Gra08}, Corollary 3.2.10]\label{lem bound Fourier coefficients}
Let $f:(\RR/\ZZ)^m\to\CC$ be a $C^k$ function and $0\neq m\in\ZZ^m$. Then the $m$-th Fourier coefficient of $f$ satisfies the bound
$$\abs{\hat f(m)}\ll\frac{S_k^f}{\norm m^k_\infty},$$
where $S_k^f$ is the Sobolev norm of $f$ of order $k$ with respect to the sup-norm.
\end{lemma}

We recall the setting of Theorem \ref{thm application intro}. Let $\TT_n=\XXc_n(\RR/\ZZ)$. Let $C\in\Mat_n(\ZZ)$ be such that $\det(C)\neq0$. Consider
$$S_C:=\left\{(C^{-t}A^t,C^{-1}D)\in\TT_n\times\TT_n\,\middle|\,\pmatrixtwo A\ast CD\in X(C)\right\}.$$

\begin{remark}
Note that if $C=mI_n$, then we have
$$S_{mI_n}=\left\{(X/m,\bar X/m)\in\TT_n\times\TT_n:X\in\XXc_n(\ZZ/m\ZZ),\ m\ndivides\det(X)\right\}.$$
\end{remark}

\begin{theorem}
Let $C\in\Mat_n(\ZZ)$ be such that $\det(C)\neq0$. Let $f:\TT_n\times\TT_n\to\CC$ be a $C^k$-function with $k\geq n(n+1)+1$. We have
$$\frac1{\abs{S_C}}\sum_{(M,N)\in S_C}f(M,N)=\int_{\TT_n\times\TT_n}f(X_1,X_2)dX_1\,dX_2+O_{n,\epsilon}\left(S_k^fc_1^{-1/2}c_n^\epsilon\right)$$
where $S_k^f$ is the Sobolev norm of $f$ of order $k$ with respect to the sup-norm and $c_1\divides\cdots\divides c_n$ are the elementary divisors of $C$.
\end{theorem}

\begin{remark}
We see that if $c_1\to\infty$ and the ratio between $c_1$ and $c_n$ stays constant, then the set $S_C$ equidistributes. This is the case if $C=mC_0$ with $C_0$ a matrix with $\det(C_0)\neq0$ and $m\to\infty$.
\end{remark}

\begin{proof}
Le $f:\TT_n\times\TT_n\to\CC$ be a $C^k$ function. We want to compute
\begin{align}\label{eq base sum over S_C}
\frac1{\abs{X(C)}}\sum_{\smmatrix A\ast CD\in X(C)}f(C^{-t}A^t,C^{-1}D)
\end{align}
For $Q,T$ half-integral symmetric matrices, we have the Fourier coefficient
$$\hat f(Q,T)=\int_{\TT_n\times\TT_n}f(Y_1,Y_2)e(-QY_1-TY_2)dY_1\,dY_2$$
and the Fourier series
$$f(X_1,X_2)=\sum_{Q,T}\hat f(Q,T)e(QX_1+TX_2).$$
By Theorem 3.2.16 in \cite{Gra08}, the series converges absolutely. Inserting this in Equation \eqref{eq base sum over S_C}, we get
$$\hat f(0,0)+\frac1{\abs{X(C)}}\sum_{(Q,T)\neq(0,0)}\hat f(Q,T)\sum_{\smmatrix A\ast CD\in X(C)}e(QC^{-t}A^t+TC^{-1}D).$$
The last sum is $K(Q,T;C)$. Using the bounds from Proposition \ref{pro final bound} and Lemma \ref{lem bound Fourier coefficients}, we get
$$\frac1{\abs{X(C)}}\sum_{(Q,T)\neq(0,0)}\hat f(Q,T)K(Q,T;C)\ll_{n,\epsilon}c_1^{-1/2}c_n^\epsilon\sum_{(Q,T)\neq(0,0)}\frac{(c_1,2Q,2T)^{3/2}}{\max\{\norm Q_\infty^k,\norm T_\infty^k\}}$$
We write $\ell=(c_1,2Q,2T)$. The sum over $Q,T$ is then
\begin{align*}
\sum_{(Q,T)\neq(0,0)}\frac{(c_1,2Q,2T)^{3/2}}{\max\{\norm Q_\infty^k,\norm T_\infty^k\}}&\ll_n\sum_{\ell\divides c_1}\sum_{(Q,T)\neq(0,0)}\frac{\ell^{3/2}}{\ell^k\max\{\norm Q_\infty^k,\norm T_\infty^k\}}\\
	&\ll_n\sum_{\ell\divides c_1}\ell^{3/2-k}\sum_{m=1}^\infty m^{n(n+1)-1-k}
\end{align*}
where we wrote $m=\max\{\norm Q_\infty,\norm T_\infty\}$. Note that the number of pair $(Q,T)$ with a fixed value $m$ is $O_n(m^{n(n+1)-1})$ since at least one coordinate must have value $m$. The two sums are uniformly bounded if $k\geq n(n+1)+1$. We deduce that
\begin{align*}
\frac1{\abs{X(C)}}\sum_{\smmatrix A\ast CD\in X(C)}f(C^{-t}A^t,C^{-1}D)&=\hat f(0,0)+O(c_1^{-1/2}c_n^\epsilon)\\
	&=\int_{T\times T}f(Y_1,Y_2)dY_1\,dY_2+O(c_1^{-1/2}c_n^\epsilon).
\end{align*}
\end{proof}

\printbibliography
\end{document}